\documentclass[10pt,a4paper,oneside,openany]{book}

\usepackage{latexsym}
\usepackage{amsfonts}
\usepackage{amssymb}
\usepackage{amsmath}
\usepackage{amsthm}
\usepackage{setspace}
\usepackage{color}
\usepackage{hyperref}
\usepackage{pdfpages}
\usepackage{paralist} 
\usepackage{etoolbox} 

\patchcmd{\thebibliography}{\chapter*}{\section*}{}{}

\newcommand{\beq}{\begin{equation}}		
\newcommand{\eeq}{\end{equation}}			
\newcommand{\beqq}{\begin{equation*}}	
\newcommand{\eeqq}{\end{equation*}}		

\newcommand{\sgn}{\text{sign}}

\newcommand{\id}{1\hspace{-0,9ex}1}

\renewcommand{\P}{\mathbb{P}}

\newcommand{\wlim}{\text{w -}\lim}

\newcommand{\tp}{\tilde{p}} 
\newcommand{\B}{\mathfrak{B}} 
\newcommand{\F}{\mathcal{F}} 
\newcommand{\A}{\mathcal{A}} 
\newcommand{\E}{\mathbb{E}}

\allowdisplaybreaks

\newtheorem{theorem}{Theorem}[section]
\newtheorem{lemma}[theorem]{Lemma}
\newtheorem{definition}[theorem]{Definition}
\newtheorem{remark}[theorem]{Remark}

\newtheorem{corollary}[theorem]{Corollary}

\newtheorem{proposition}[theorem]{Proposition}

\newtheorem{stepp}{\noindent\bf{Step}}

{\begin{stepp}\rm}{\rm \end{stepp}}

\renewcommand*{\thesection}{\arabic{section}}

\begin{document}
\pagestyle{headings} \thispagestyle{headings} \thispagestyle{empty}

\noindent\rule{15.812cm}{0.4pt}
\begin{center}
\textsc{A randomized weighted $p$-Laplacian evolution equation with Neumann boundary conditions}
\end{center}
\begin{center}
	by
\end{center} 
\begin{center}
	\textsc{Alexander Nerlich\footnote{Author's Affiliation: Ulm University, 89081 Ulm, Helmholtzstr. 18 }\footnote{Author's E-Mail: alexander.nerlich@uni-ulm.de}\footnote{Author's ORCID: 0000-0001-7823-0648}}
\end{center}
\noindent\rule{15.812cm}{0.4pt}
\vspace{0.4cm}
\pagestyle{myheadings} 
\begin{center} 
	{\large ABSTRACT} 
\end{center}

\noindent The purpose of this paper is to show that the randomized weighted $p$-Laplacian evolution equation given by
\begin{align}
\label{eveqrand}
\begin{cases}  U^{\prime}(t)(\omega) =\text{div} \left( g(\omega) | \nabla U(t)(\omega)| ^{p-2}\nabla U(t)(\omega) \right) \text{  on }  S, \\ 
 g(\omega)| \nabla U(t)(\omega)|^{p-2}\nabla U(t)(\omega)\cdot \eta=0 \text{  on }  \partial S, \\
U(0)(\omega)=u(\omega),\end{cases}
\end{align}
for $\P$-a.e. $\omega \in \Omega$ and a.e. $t \in (0,\infty)$ admits a unique strong solution and to determine asymptotic properties of this solution.

\section{Introduction} 
The existence of a unique strong solution of
\begin{align}
\label{eveqdeterministic}
\begin{cases} V^{\prime}(t) = \text{div} \left( \gamma | \nabla V(t)| ^{p-2}\nabla V(t) \right) & \text{on }  S\\ 
 \gamma|\nabla V(t)|^{p-2}\nabla V(t)\cdot \eta=0 & \text{on } \partial S,\\
V(0)=v,\end{cases}
\end{align}
for a.e. $t \in (0,\infty)$, has been proven by F. Andreu, J.M. Maz\'{o}n, J. Rossi and J. Toledo in \cite{main}. In addition, the importance of the PDE (\ref{eveqdeterministic}) to the evolution of fluvial landscapes has been discussed by  B. Birnir and J. Rowlett in \cite{birnirtheory}. Moreover, the asymptotic properties of the solution of (\ref{eveqdeterministic}) have been studied in \cite{ich}.\\ 

The purpose of this paper is to study a random version of the PDE (\ref{eveqdeterministic}); more precisely: The weight function $\gamma:S \rightarrow (0,\infty)$ occurring in (\ref{eveqdeterministic}) is replaced by a (vector-valued) random variable $g:\Omega \rightarrow L^{1}(S)$. Hereby $(\Omega,\mathcal{F},\P)$ denotes a complete probability space, $S \subseteq \mathbb{R}^{n}$, where $n \in \mathbb{N}\setminus \{1\}$, is a sufficiently regular set, $\eta$ is the unit outer normal on $\partial S$ and  $p \in (1,\infty)\setminus\{2\}$.\\
The weight function $g$ being random of course implies that the solution has to be random as well. Consequently, it is appropriate to assume that the initial value is no longer deterministic but also a random quantity. Therefore it is natural to consider the randomized PDE (\ref{eveqrand}) as the PDE corresponding to (\ref{eveqdeterministic}) for a random weight function.\\ 
From an applied point of view, the weight function $\gamma$ in (\ref{eveqdeterministic}) models a stationary water depth which occurs due to rain, on the landscape $v$. The motivation for considering the randomized PDE (\ref{eveqintro}) is that this water depth might be not precisely known, which makes it reasonable to view it as a random quantity.\\
Hereby, the assumptions made on $g(\omega)$ for a given $\omega \in \Omega$ are, due to technical reasons, actually stronger than those made on $\gamma$ in \cite{main}.\\

It seems important to point out that this paper is not concerned with any kind of stochastic differential equation. The noise occurring in this paper's setting does not come from integrating with respect to Brownian motions or other stochastic processes, but originates from the random weight function $g$ and the random initial value $u$.\\

\section{The main results}

The existence of a unique solution of (\ref{eveqrand}) will be obtained by means of the nonlinear semigroup theory; more precisely:\\
First of all a single-valued operator $A:D(A)\rightarrow L^{1}(\Omega;L^{1}(S))$ is introduced, which is defined such that a continuous function $U:[0,\infty) \rightarrow L^{1}(\Omega;L^{1}(S))$, with $U \in W^{1,1}_{\text{Loc}}((0,\infty);L^{1}(\Omega;L^{1}(S)))$, fulfilling
\begin{align}
\label{eveqintro}
-U^{\prime}(t)=AU(t)\text{, for a.e. }t \in (0,\infty),~U(0)=u,
\end{align}
is also a function fulfilling (\ref{eveqrand}). Afterwards it will be established that $A$ is accretive, fulfills a certain range condition and has dense domain. Unfortunately, $A$ is not m-accretive, but its closure, denoted by $\A$, is. Consequently, one obtains that 
\begin{align}
\label{eveqintro2}
0 \in U^{\prime}(t)+\A U(t)\text{, for a.e. }t \in (0,\infty),~U(0)=u,
\end{align}
has a unique mild solution for arbitrary initial values $u \in L^{1}(\Omega;L^{1}(S))$. Thereafter, it will be established that this mild solution is even a strong solution. Moreover, it will be proven that the uniquely determined strong solution of (\ref{eveqintro2}) also fulfills (\ref{eveqintro}), and therefore (\ref{eveqrand}), provided that in addition $\P( u \in L^{\infty}(S))=1$.\\ 

For a given $u \in L^{1}(\Omega;L^{1}(S))$, let $T(\cdot)u: [0,\infty) \rightarrow L^{1}(\Omega;L^{1}(S))$ be the strong solution of (\ref{eveqintro2}) and let $T(\cdot,\gamma)v :[0,\infty) \rightarrow L^{1}(S)$ be the strong solution of its deterministic counterpart, where $v \in L^{1}(S)$ and the weight function $\gamma \in L^{1}(S)$ fulfills certain regularity conditions.\\
It will be proven that $(T(t)u)(\omega)=T(t,g(\omega))u(\omega)$ for $\P$-a.e. $\omega \in \Omega$. The latter makes it possible to transfer many results from \cite{ich} to the current setting; more precisely:\\ 
Firstly, $T(\cdot)u$ conserves mass with probability one, i.e. for a given $t \in [0,\infty)$ one has\linebreak $\overline{((T(t)u)(\omega))} =\overline{(u(\omega))}$ for $\P$-a.e. $\omega \in \Omega$, every $u \in L^{1}(\Omega;L^{1}(S))$ and $t\geq 0$, where $\overline{(v)}:= \frac{1}{\lambda(S)}\int \limits_{S}v d\lambda$ is the average of any $v \in L^{1}(S)$; hereby $\lambda$ is the Lebesgue measure on $\mathbb{R}^{n}$.\\
Secondly, one has for a given $t \in (0,\infty)$ and  $u \in L^{1}(\Omega;L^{1}(S))$, with $\P(u \in L^{2}(S))=1$, that
\begin{align}
\label{introeqasymp}
||(T(t)u)(\omega)-\overline{(u(\omega))}||_{L^{1}(S)} \leq c_{1} \Delta_{u}(\omega)^{\frac{1}{p}}
\left(\frac{1}{t}\right)^{\frac{1}{p}},~\text{$\P$-a.e. } \omega \in \Omega,
\end{align}
where $\Delta_{u}:\Omega \rightarrow \mathbb{R}$, with $\Delta_{u}(\omega):=||u(\omega)-\overline{(u(\omega))}||^{2}_{L^{2}(S)}$. Hereby it is actually possible to explicitly determine $c_{1} \geq 0$.\\ 
Finally, if $p>n$, and if $u \in L^{1}(\Omega;L^{1}(S))$ is such that $\P(u \in L^{p}(S))=1$, then one has
\begin{align}  
\label{introeqinf}
||(T(t)u)(\omega)-\overline{(u(\omega))}||_{L^{\infty}(S)} \leq c_{2} \Delta_{u}(\omega)^{\frac{1}{p}} \left(\frac{1}{t}\right)^{\frac{1}{p}},~\text{$\P$-a.e. } \omega \in \Omega,
\end{align}
where $c_{2}$ can be determined explicitly.\\
Relation (\ref{introeqinf}) is a very strong regularity result. Proving it is actually fairly simple, since one can relate $T(t)u$ to its deterministic counterpart.\\

There are two asymptotic results proven in this paper which are not direct consequences of the results in \cite{ich}.\\
Firstly, one has 
\begin{align}
\label{introeqasymp2}
\lim \limits_{t \rightarrow \infty} T(t)u=\overline{(u)},\text{ in } L^{q}(\Omega;L^{q}(S)).
\end{align}
for any $u \in L^{q}(\Omega;L^{q}(S))$, $q \in [1,\infty)$.\\
Secondly, it is possible to derive an upper bound for the tail function of $||T(t)u-\overline{(u)}||^{2}_{L^{2}(S)}$, assuming that $p \in [\frac{2n}{n+2},2)\setminus \{1\}$; more precisely: Let $r \in [1,\infty)$ and assume $||u||^{4r}_{L^{2}(S)} \in L^{1}(\Omega)$. Then one has
\begin{align}
\label{introfuckingnicebound1}
\P \left( \int \limits_{S} (T(t)u-\overline{(u)})^{2}d\lambda > \alpha \right) \leq \left(\frac{1}{t}\right)^{r} \frac{2}{\log(\alpha+1)} c_{3}~ \left(\E(\Delta_{u})\E((1+\Delta_{u})^{2r})\right)^{\frac{1}{2}},  
\end{align}
for any $\alpha,~t \in (0,\infty)$; and if there is even an $\varepsilon>0$ such that $e^{\varepsilon ||u||^{2}_{L^{2}(S)}} \in L^{1}(\Omega)$, one has
\begin{align}
\label{introfuckingnicebound2}
\P \left( \int \limits_{S} (T(t)u-\overline{(u)})^{2}d\lambda > \alpha \right) \leq \exp\left(-t^{\frac{1}{2}}c_{4}\right)\frac{2 \exp\left(\frac{\varepsilon}{2}\right)}{\log(\alpha+1)}\left(\E(\Delta_{u})\E \left(\exp \left(\varepsilon\Delta_{u}\right)\right)\right)^{\frac{1}{2}},
\end{align} 
for any $\alpha,~t \in (0,\infty)$.\\
Hereby $c_{3},~c_{4} \geq 0$ are again constants which can be determined explicitly.\\
One should note that if $n=2$, which is from an applied point of view the interesting case, one can apply either (\ref{introeqinf}) or (\ref{introfuckingnicebound1}), resp. (\ref{introfuckingnicebound2}), given that the initial is sufficiently integrable.\\

This paper is structured as follows: Section \ref{basicdef} contains all assumptions, notations and basic definitions which are needed throughout this paper. Section \ref{section_carc} (Section \ref{section_ss}, resp.) deals with the existence and uniqueness of mild (strong, resp.) solutions of (\ref{eveqintro2}). The assertions (\ref{introeqasymp})-(\ref{introeqasymp2}) are established in Section \ref{section_gar}. Finally, Section \ref{section_desmallp} deals with the tail function bounds (\ref{introfuckingnicebound1}) and (\ref{introfuckingnicebound2}).\\
Moreover, this article contains two appendices. The first answers some technical measurability questions which occur while defining $A$ and the second one provides some delicate results about the deterministic counterpart of $A$. Those are mostly needed to prove the existence and uniqueness of mild solutions of (\ref{eveqintro2}).

\section{Assumptions, notations and preliminary results}
\label{basicdef}
Some notational preliminaries are in order: For any $m$-dimensional Borel measurable set $B$, where $m \in \mathbb{N}$, $\mathfrak{B}(B)$ denotes the Borel $\sigma$-algebra on this set. Moreover, for any measure $\mu:\mathfrak{B}(B) \rightarrow [0,\infty]$ and any $q \in [1,\infty]$, $L^{q}(B,\mu;\mathbb{R}^{m})$ denotes the usual Lebesgue spaces and $||\cdot||_{L^{q}(B,\mu;\mathbb{R}^{m})}$ denotes the canonical norm on these spaces.\\
If $m=1$, $L^{q}(B,\mu;\mathbb{R}^{m})$ is abbreviated by $L^{q}(B,\mu)$ and if $\mu$  is the Lebesgue measure, $L^{q}(B)$ is written. Of course the analogous convention applies to $||\cdot||_{L^{q}(B,\mu;\mathbb{R}^{m})}$.\\

Now let $B \subseteq \mathbb{R}^{m}$ be open. $L^{1}_{\text{Loc}}(B;\mathbb{R}^{m})$ denotes the space of locally Lebesgue integrable functions $f:B\rightarrow \mathbb{R}^{m}$. If $m=1$, then $L^{1}_{\text{Loc}}(B)$ is written.\\ 
Moreover, $C^{\infty}(\overline{B})$ denotes the space of all infinitely often continuously differentiable function, such that the function and all of its partial derivatives can be extended continuously to the boundary of $B$ and $C^{\infty}_{c}(B)$ denotes the space of all functions $\varphi \in C^{\infty}(\overline{B}) $ which have compact support contained in $B$. Moreover, $C^{\infty}(\overline{B};\mathbb{R}^{m})$, $C^{\infty}_{c}(\overline{B};\mathbb{R}^{m})$ resp., denotes the space of all functions $\varphi=(\varphi_{1},...,\varphi_{m})$ such that $\varphi_{j} \in C^{\infty}(\overline{B})$ for all $j=1,...,m$, $\varphi_{j} \in C^{\infty}_{c}(\overline{B})$ for all $j=1,...,m$, resp.\\
Moreover, $W^{1,1}_{\text{Loc}}(B)$ denotes the space of weakly differentiable functions and $\nabla f$ denotes the weak derivative of any $f \in W^{1,1}_{\text{Loc}}(B)$. In addition, $W^{1,q}(B)$ denotes the Sobolev space of once weakly differentiable functions, such that the function and all of its weak derivatives are in $L^{q}(B)$.\\ 
For any Banach space $(X,||\cdot||_{X})$, introduce
\begin{align*}
W^{1,1}_{\text{Loc}}( (0,\infty);X):=\{f:(0,\infty)\rightarrow X|~f \text{ is locally absolutely continuous and differentiable a.e.}\} 
\end{align*}
and
\begin{align*}
C( [0,\infty);X):=\{f:[0,\infty)\rightarrow X|~f \text{ is continuous}\}.
\end{align*}
In addition, for any $f \in W^{1,1}_{\text{Loc}}((0,\infty);X) $ the function $f^{\prime}$ denotes the almost everywhere existing derivative of $f$.\\
Moreover, $\mathfrak{B}(X)$ denotes the Borel $\sigma$-algebra on $X$ and $2^{X}$ denotes its power set.\\
For a complete probability space $(\Omega,\mathcal{F},\P)$, introduce
\begin{align*}
L^{q}(\Omega;X):=\{f:\Omega\rightarrow X|~f\text{ is } \mathcal{F}-\mathfrak{B}(X)-\text{meas. and } \int \limits_{\Omega} ||f(\omega)||^{q}_{X}d\P(\omega) <\infty\},~\forall q\in [1,\infty)
\end{align*}
and let $||\cdot||_{L^{q}(\Omega;X)}$ denote its canonical norm. This is abbreviated by $L^{q}(\Omega)$, if $X=\mathbb{R}$.\\

Finally, $\lambda$ denotes the Lebesgue measure, $x\cdot y$ the canonical inner product between two vectors $x,y \in \mathbb{R}^{m}$ and $|\cdot|$ the Euclidean norm on $\mathbb{R}^{m}$.\\ 

Here and in all that follows let $n \in \mathbb{N} \setminus \{1\}$ be arbitrary but fixed, let $p \in (1,\infty) \setminus \{2\}$ and let $\tilde{p}$ denote its H\"older conjugate, i.e. $\tilde{p} \in (1,\infty)$ is such that $\frac{1}{p}+\frac{1}{\tilde{p}}=1$.  Additionally, introduce $\emptyset \neq S \subseteq \mathbb{R}^{n}$ and assume that it is an open, bounded, connected set of class $C^{1}$. Moreover, $(\Omega,\mathcal{F},\P)$ denotes a complete probability space.\\ 
In addition, introduce for any $0 < \varepsilon_{1} < \varepsilon_{2}< \infty$ the set
\begin{align*}
L^{1}_{\varepsilon_{1},\varepsilon_{2}}(S):=\{ f \in L^{1}(S): \varepsilon_{1} \leq f \leq \varepsilon_{2} \text{ a.e. on S}\}
\end{align*}
and note that $L^{1}_{\varepsilon_{1},\varepsilon_{2}}(S)$ is closed w.r.t. $||\cdot||_{L^{1}(S)}$. This yields that $L^{1}_{\varepsilon_{1},\varepsilon_{2}}(S) \in \mathfrak{B}(L^{1}(S))$.\\
Moreover, introduce analogously
\begin{align*}
L^{1}_{\varepsilon_{1},\varepsilon_{2}}(\Omega;L^{1}(S)):=\{ f \in L^{1}(\Omega;L^{1}(S)):~\P(\{\omega\in \Omega: f(\omega) \in  L^{1}_{\varepsilon_{1},\varepsilon_{2}}(S) \})=1\}.
\end{align*}
Finally, let  $0< g_{1}\leq g_{2} <\infty$ be constants and let $g \in L^{1}_{g_{1},g_{2}}(\Omega;L^{1}(S))$.

\begin{remark} It will be necessary to consider expressions of the form $|x|^{p-2}x$, resp. $|x|^{p-2} x\cdot y$ for given $x,y\in\mathbb{R}^{n}$. And the last needs to be considered even if $x=0$.\\ 
Notation will be slightly abused by setting $|x|^{p-2}x:=0$ and $|x|^{p-2} x\cdot y :=0$ if $x=0$ for any $y \in \mathbb{R}^{n}$.\\
This is justified because one verifies immediately that the mappings $ \mathbb{R}^{n}\setminus \{0\} \ni x \mapsto |x|^{p-2}x$ and $ \mathbb{R}^{n}\setminus \{0\} \ni x \mapsto |x|^{p-2} x\cdot y $ can be extended continuously by $0$ in $x=0$.
\end{remark}

\begin{remark} Some severe measurability issues will occur throughout this paper. The reader who is not too familiar with measurability of vector-valued mappings is referred to \cite[Chapter 5]{meas}, for a detailed introduction to this concept. Particularly, one should note that the vector-valued mappings occurring here act on $L^{1}(S)$ which is well known to be separable. Moreover, it is common knowledge that in this case the notions of measurability (in the usual measure-theoretic sense), strong measurability and weak measurability are equivalent, cf. \cite[Theorem 5.8 and Theorem 5.9]{meas}. 
\end{remark}

The following operator plays a crucial role in this paper. Some questions concerning measurability occur during its definition. Answers to these  question can be found in Appendix \ref{appendix_A}. 

\begin{definition}\label{defA} Let $A:D(A)\rightarrow 2^{L^{1}(\Omega;L^{1}(S))}$ be such that $(f,\hat{f}) \in A$ if and only if the following assertions hold.
\begin{enumerate}
\item $f \in L^{1}(\Omega;L^{1}(S))$ and $P(f \in W^{1,p}(S)\cap L^{\infty}(S))=1$. 
\item $\hat{f} \in L^{1}(\Omega;L^{1}(S))$.
\item $\P\left( \int \limits_{S}g| \nabla f|^{p-2} \nabla f\cdot \nabla \varphi d\lambda = \int \limits_{S}\hat{f}\varphi d \lambda,~\forall \varphi \in W^{1,p}(S)\cap L^{\infty}(S) \right)=1 $.
\end{enumerate} 
 Moreover, let $\A:D(\A)\rightarrow 2^{L^{1}(\Omega;L^{1}(S))}$ be the closure of $A$, i.e. $(f,\hat{f}) \in \A$ if and only if there is a sequence $((f_{m},\hat{f}_{m}))_{m \in \mathbb{N}} \subseteq A$ such that
 \begin{align*}
 \lim \limits_{m \rightarrow \infty} (f_{m},\hat{f}_{m})=(f,\hat{f}),\text{ in } L^{1}(\Omega;L^{1}(S))^{2}.
 \end{align*} 
\end{definition}

\begin{definition} Let $\gamma \in L^{1}_{g_{1},g_{2}}(S)$ and introduce the single-valued operator $a(\gamma):D(a(\gamma))\rightarrow L^{1}(S)$ by: $(f,\hat{f}) \in a(\gamma)$ if and only if the following assertions hold.
	\begin{enumerate}
		\item $f \in W^{1,p}(S) \cap L^{\infty}(S)$. 
		\item $\hat{f} \in L^{1}(S)$.
		\item $\int \limits_{S}  \gamma|\nabla f|^{p-2}\nabla f\cdot \nabla \varphi  d \lambda = \int \limits_{S} \hat{f} \varphi d \lambda$ for all $\varphi\in W^{1,p}(S)\cap L^{\infty}(S)$.
	\end{enumerate}
	Moreover, let $\mathfrak{a}(\gamma):D(\mathfrak{a}(\gamma))\rightarrow L^{1}(S)$ be the closure of $a(\gamma)$, i.e. $(f,\hat{f})\in \mathfrak{a}(\gamma)$ if and only if there is a sequence $((f_{m},\hat{f}_{m}))_{m \in \mathbb{N}} \subseteq a(\gamma)$ such that
	\begin{align*}
	\lim \limits_{m \rightarrow \infty} (f_{m},\hat{f}_{m})=(f,\hat{f}),\text{ in } L^{1}(S)^{2}.
	\end{align*}
\end{definition}

As $a(\gamma)$ is single-valued, it will either be written $(f,\hat{f}) \in a (\gamma)$ or $\hat{f}=a(\gamma)f$, depending on which of both notations is more convenient for the setting.\\ 

The operators $a(\gamma)$ and $\mathfrak{a}(\gamma)$ are the operators which are used in \cite{main} to prove the existence of unique solutions of (\ref{eveqdeterministic}), cf. \cite[Theorem 3.7]{main}.\\ 
An explicit description of $\mathfrak{a}(\gamma)$ can be found in \cite[Prop. 3.6]{main}. But this description is not needed for our purposes.\\

\begin{remark}\label{remarkca} Let $J_{0}$ denote the space of all lower-semicontinuous and convex functions, $j: \mathbb{R} \rightarrow [0,\infty]$, fulfilling $j(0)=0$. Moreover, for two functions $f_{1},~f_{2} \in L^{1}(S)$ it will be written $f_{1}<<f_{2}$, whenever
	\begin{align*}
	\int \limits_{S} j \circ f_{1}d \lambda \leq \int \limits_{S} j \circ f_{2} d \lambda,~\forall j \in J_{0}.
	\end{align*}
	Finally an operator $B:D(B)\rightarrow 2^{L^{1}(S)}$ is called completely accretive if
	\begin{align*}
	f-h << f-h +\alpha (\hat{f}-\hat{h}),
	\end{align*}
	for all $\alpha \in (0,\infty)$ and $(f,\hat{f}),~(h,\hat{h}) \in B$.\\
	Note the following: If $f_{1},f_{2} \in L^{1}(S)$ are such that $f_{1}<<f_{2}$, one has $||f_{1}||_{L^{q}(S)}\leq||f_{2}||_{L^{q}(S)}$, for every $q \in [1,\infty]$, provided $f_{2} \in L^{q}(S)$.
\end{remark}

Besides some technical advantages, the concept of complete accretivity also offers a (fairly straightforward) way to prove differentiability almost everywhere of mild solutions of evolution equations, even if the evolution equation is framed in a Banach space which does not have the Radon-Nikodym property, cf. \cite[Theorem 4.2 and 4.4]{cao} . The latter is the main reason why this concept is frequently used when it comes to nonlinear evolution equations.\\
The reader is referred to \cite{cao} for a detailed treatment of complete accretivity.

\begin{remark}Besides the concept of complete accretivity, the notion of accretivity and m-accretivity will be needed as well. An operator $B:D(B)\rightarrow 2^{X}$, where $(X,||\cdot||_{X})$ is an arbitrary Banach space, is called accretive if
	\begin{align*}
		||f-h||_X \leq ||f-h+\alpha(\hat{f}-\hat{h})||_X
	\end{align*}
for all $(f,\hat{f})$,~$(h,\hat{h})\in B$ and $\alpha>0$.\\
Moreover, $B$ is called m-accretive if it is accretive and
\begin{align*}
	R(Id+A)=X.
\end{align*}
Now fix $\beta \in (0,\infty)$ and let $B:D(B) \rightarrow 2^{X}$ be an accretive operator, then it is clear that there is for every $h \in R(Id+\beta B)$ precisely one pair $(f,\hat{f}) \in B$ such that $h=f+\beta\hat{f}$. Consequently, one can introduce $(Id+\beta B)^{-1}:R(Id+\beta B)\rightarrow D(B)$, where $(Id+\beta B)^{-1}h$ is precisely the element $f \in D(B)$, such that there is an $\hat{f} \in Bf$ with $h=f+\beta\hat{f}$.\\
The mapping $(Id+\beta B)^{-1}$ is called the resolvent of $B$.
\end{remark}

This section concludes by introducing two more spaces needed in this paper. In that what follows, $\tau_{k}:\mathbb{R}\rightarrow \mathbb{R}$, where $k \in (0,\infty)$, denotes the standard truncation function, that is,
\begin{align*} 
\tau_{k}(s):=\begin{cases} s, & \text{if } |s|< k
\\ k\sgn(s), & \text{if } |s|\geq k.  \end{cases}
\end{align*} 
For $f: S\rightarrow \mathbb{R}$, $\tau_{k}(f):S \rightarrow \mathbb{R}$ is defined by $\tau_{k}(f):=\tau_{k}(f(\cdot))$ and for $f:\Omega \rightarrow L^{1}(S)$, $\tau_{k}(f):\Omega \rightarrow L^{1}(S)$ is defined by $\tau_{k}(f)(\omega):=\tau_{k}(f(\omega))$.\\
Moreover, introduce
\begin{align*}
\tau(L^{1}(\Omega;L^{1}(S))):=\{ \tau_{k}(f)|~f \in L^{1}(\Omega;L^{1}(S)),~k \in (0,\infty)\}  
\end{align*} 
and
\begin{align*}
L^{1,\infty}(\Omega;L^{1}(S)):=\{f \in L^{1}(\Omega;L^{1}(S))|~\P(f \in L^{\infty}(S))=1\}.
\end{align*}

\section{Mild solutions of the randomized weighted $p$-Laplacian evolution equation with Neumann boundary conditions}
\label{section_carc}

The following result serves as a guideline for this section and is taken from \cite{BenilanBook} which is a very comprehensive book on nonlinear semigroups and evolution equations.
\begin{theorem}\label{lemmamsg} Let $(X,||\cdot||)$ be a Banach space and $B: D(B) \rightarrow 2^{X}$ be an m-accretive operator. Moreover, assume that $D(B)$ is dense in $(X,||\cdot||)$. Then the evolution equation
\begin{align*}
 0 \in U^{\prime}(t)+ B U(t)\text{, for a.e. } t \in (0,\infty),~ U(0)=u,
\end{align*}
has for any $u \in X$ precisely one mild solution $U \in C([0,\infty);X)$.
\end{theorem}
\begin{proof} See \cite[Theorem 4.6]{BenilanBook} for the existence and \cite[Theorem 6.5]{BenilanBook} for the uniqueness.
\end{proof}

Consequently, as the headline suggests, the purpose of this section is to prove that $\A$ is m-accretive and has dense domain.\\
For proving this, some technical properties of $a(\gamma)$ have to be established. These technical results and their proofs have been moved to Appendix \ref{appendix_b}. Moreover, the denseness of the spaces $\tau(L^{1}(\Omega;L^{1}(S)))$ and $L^{1,\infty}(\Omega;L^{1}(S))$ is also proven in Appendix \ref{appendix_b}.\\ Particularly, none of the proofs in Appendix \ref{appendix_b} relies on any result in this section.\\

The Lemmata \ref{lemmawelldefined}-\ref{propclosaarelation} are essentially a collection of useful properties of the considered operators. These results are on the one hand of extreme importance for the next sections, and on the other hand, they build the path to a fundamental estimate which yields particularly the accretivity of $\A$, see Proposition \ref{lemmacaversion}. \\ 
Lemma \ref{measlemma}, together with Appendix \ref{appendix_b}, brings us in the position to prove that $\A$ is m-accretive, which is achieved in Theorem \ref{rctheorem}. Finally, it will be established that $A$, and a fortiori also $\A$, has dense domain, which then implies the first main result of this paper, namely the existence of unique mild solutions of (\ref{eveqintro2}).

\begin{lemma}\label{lemmawelldefined} The operator $A$ is single-valued.
\end{lemma}
\begin{proof} Let $(f,\hat{f}),~(f,\tilde{f}) \in A$. Then one has
\begin{align*}
\int \limits_{S} (\hat{f}-\tilde{f})\varphi d \lambda =0,~\forall \varphi \in W^{1,p}(S) \cap L^{\infty}(S)
\end{align*}
with probability one. Consequently $\tilde{f}=\hat{f}$ a.e. on $S$ with probability one, i.e. $\hat{f}=\tilde{f}$ as elements of $L^{1}(\Omega,L^{1}(S))$.
\end{proof} 

\begin{lemma}\label{lemmaAarel} Let $f,~\hat{f} \in L^{1}(\Omega;L^{1}(S))$. The following assertions are equivalent.
\begin{enumerate}
\item  $(f,\hat{f})\in A$
\item  $\P\left(\big\{\omega \in \Omega:~(f(\omega),\hat{f}(\omega))\in a(g(\omega))\big\}\right)=1$
\end{enumerate} 
\end{lemma}
\begin{proof} Let $f,~\hat{f} \in L^{1}(\Omega;L^{1}(S))$. Then one has 
\begin{eqnarray*}
	& & ~ 
	\{\omega:~(f(\omega),\hat{f}(\omega))\in a(g(\omega))\} \\
	& = & ~ \{ \omega:~f \in W^{1,p}(S) \cap L^{\infty}(S),~\int \limits_{S}  g| \nabla f|^{p-2}\nabla f\cdot \nabla \varphi  d \lambda = \int \limits_{S} \hat{f} \varphi d \lambda,~\forall\varphi\in W^{1,p}(S)\cap L^{\infty}(S)\},
\end{eqnarray*}
which yields, by invoking Lemma \ref{lemmawelldefined1} and Lemma \ref{lemmawelldefined2} that
the event $\{\omega:~(f(\omega),\hat{f}(\omega))\in a(g(\omega))\}$ is measurable. Moreover, the former equality also yields the equivalence of i) and ii).
\end{proof}

\begin{lemma}\label{lemmaAclosarel} Let $(f,\hat{f}) \in \A$. Then one has
\begin{align*}
\P\left(\big\{\omega \in \Omega:~(f(\omega),\hat{f}(\omega)) \in \mathfrak{a}(g(\omega))\big\}\right)=1.
\end{align*}
\end{lemma}
\begin{proof} As $(f,\hat{f})\in \A$, there is, by passing to a subsequence if necessary, a sequence $((f_{m},\hat{f}_{m}))_{m\in \mathbb{N}} \subseteq A$ such that
\begin{align*}
\lim \limits_{m \rightarrow \infty} (f_{m}(\omega),\hat{f}_{m}(\omega)) = (f(\omega),\hat{f}(\omega)),~\text{for $\P$-a.e. } \omega \in \Omega,~\text{in } L^{1}(S)^{2}.
\end{align*}
Consequently, Lemma \ref{lemmaAarel} yields that one has, up to a $\P$-nullset
\begin{eqnarray*}
	& & ~
\{\omega \in \Omega:~(f(\omega),\hat{f}(\omega)) \in \mathfrak{a}(g(\omega))\big\}\\
& = & ~ \{\omega \in \Omega:~\exists (F_{m}(\omega),\hat{F}_{m}(\omega))_{m \in \mathbb{N}} \subseteq a(g(\omega)),~\lim \limits_{m \rightarrow \infty} (F_{m}(\omega),\hat{F}_{m}(\omega)) = (f(\omega),\hat{f}(\omega)),~\text{in } L^{1}(S)^{2}\big\}\\
& \supseteq & ~ \{\omega \in \Omega:~\lim \limits_{m \rightarrow \infty} (f_{m}(\omega),\hat{f}_{m}(\omega)) = (f(\omega),\hat{f}(\omega)),~\text{in } L^{1}(S)^{2}\big\}\\
& = & ~ \Omega.
\end{eqnarray*}
Hence $\{\omega \in \Omega:~(f(\omega),\hat{f}(\omega)) \in \mathfrak{a}(g(\omega))\big\}$ is, up to a $\P$-nullset, equal to $\Omega$. Therefore this event is $\F$-measurable, because $(\Omega,\F,\P)$ is complete, and occurs with probability one.
\end{proof}

\begin{lemma}\label{propclosaarelation} Let $(f,\hat{f}) \in \A$ and assume $f \in L^{1,\infty}(\Omega;L^{1}(S))$. Then $(f,\hat{f}) \in A$.
\end{lemma}
\begin{proof} Invoking Lemma \ref{lemmaAclosarel} yields that $(f(\omega),\hat{f}(\omega)) \in \mathfrak{a}(g(\omega))$ for $\P$-a.e. $\omega \in \Omega$. As also $f(\omega) \in L^{\infty}(S)$ for $\P$-a.e. $\omega \in \Omega$, one has, by virtue of \cite[Lemma 3.1]{ich} that $(f(\omega),\hat{f}(\omega)) \in a(g(\omega))$ for $\P$-a.e. $\omega \in \Omega$. This yields the claim by Lemma \ref{lemmaAarel}. 
\end{proof}

\begin{proposition}\label{lemmacaversion} Let $(f,\hat{f}),~(h,\hat{h}) \in \A$, $\alpha \in (0,\infty)$, $q \in [1,\infty]$ and assume that $f-h+\alpha (\hat{f}-\hat{h}) \in L^{q}(S)$ with probability one. Then one has
\begin{align}
\label{lemmaaccreeq}
\P \left(\big\{\omega \in \Omega:~||f(\omega)-h(\omega)||_{L^{q}(S)}\leq ||f(\omega)-h(\omega)+\alpha (\hat{f}(\omega)-\hat{h}(\omega))||_{L^{q}(S)} \big\} \right) =1.
\end{align}
Particularly, $A$ as well as $\A$, are accretive.
\end{proposition}
\begin{proof} One has, by virtue of Lemma \ref{lemmaAclosarel}, that $(f(\omega),\hat{f}(\omega)),(h(\omega),\hat{h}(\omega)) \in \mathfrak{a}(g(\omega))$ for $\P$-a.e. $\omega \in \Omega$.\\
Consequently, (\ref{lemmaaccreeq}) follows by invoking Remark \ref{remarkca} and Lemma \ref{aprop}.i).\\ 
Moreover, (\ref{lemmaaccreeq}) yields particularly that $||f-h||_{L^{1}(\Omega;L^{1}(S))}\leq ||f-h+\alpha(\hat{f}-\hat{h})||_{L^{1}(\Omega;L^{1}(S))}$, i.e. $\A$ is accretive. This obviously implies that $A$ is accretive as well.
\end{proof} 

\begin{lemma}\label{measlemma} Assume that $g$ is simple, i.e. there is an $m \in \mathbb{N}$, $\gamma_{1},...,\gamma_{m} \in L^{1}(S)$ and disjoint sets $\Omega_{1},...,\Omega_{m} \in \mathcal{F}$, such that $\bigcup \limits_{k=1}\limits^{m}\Omega_{k}=\Omega$ and
	\begin{align}
	\label{measlemmaeq1}
	g(\cdot)= \sum \limits_{k=1} \limits^{m} \gamma_{k} \id_{\Omega_{k}}(\cdot).
	\end{align}
	Moreover, let $h \in L^{1,\infty}(\Omega;L^{1}(S))$. Then the mapping defined by
	\begin{align*}
	\Omega \ni \omega \mapsto (Id+a(g(\omega)))^{-1}h(\omega)
	\end{align*}
	is $\mathcal{F}$-$\mathfrak{B}(L^{1}(S))$-measurable.
\end{lemma}
\begin{proof} Let $h \in L^{1,\infty}(\Omega;L^{1}(S))$ be arbitrary but fixed and assume that $g$ is given by (\ref{measlemmaeq1}). Moreover, assume w.l.o.g. that none of the $\Omega_{k}$ is a $\P$-nullset.\\ 
	Since $g_{1} \leq g \leq g_{2}$ a.e. on $S$ with probability one, it is clear that $g_{1} \leq \gamma_{k} \leq g_{2}$ a.e. on $S$ for each $k=1,...,m$.\\
	Moreover, as $h(\omega) \in L^{\infty}(S)$ for $\P$-a.e. $\omega \in \Omega$, Lemma \ref{aprop}.ii) yields that the mapping $\varphi:\Omega\rightarrow L^{1}(S)$ defined by
	\begin{align*}
	\varphi(\omega):=(Id+a(g(\omega)))^{-1}h(\omega),~\text{for $\P$-a.e. } \omega \in \Omega
	\end{align*}
	is well-defined.\\ 
	For a given $k \in \{1,...,m\}$ and all $\omega \in \Omega_{k}$ except for a $\P$-nullset, one has $\varphi(\omega)=(Id+a(\gamma_{k}))^{-1}h(\omega)$. Consequently, Lemma \ref{aprop}.iii) yields
	\begin{align*}
	\varphi(\omega)= \sum \limits_{k=1} \limits^{m} \id_{\Omega_{k}}(\omega)(Id+a(\gamma_{k}))^{-1}h(\omega) = \sum \limits_{k=1} \limits^{m} \id_{\Omega_{k}}(\omega) (Id+\mathfrak{a}(\gamma_{k}))^{-1}h(\omega),~\text{for $\P$-a.e. } \omega \in \Omega.
	\end{align*}
	Moreover, $(Id+\mathfrak{a}(\gamma_{k}))^{-1}$ is $L^{1}(S)$-continuous for all $k=1,...,m$, see Lemma \ref{aprop}.iv). Hence, the mapping $(Id+\mathfrak{a}(\gamma_{k}))^{-1}$ is $\mathfrak{B}(L^{1}(S))-\mathfrak{B}(L^{1}(S))$-measurable. Since $h \in L^{1}(\Omega;L^{1}(S))$ it follows that $\Omega \ni \omega \mapsto (Id+\mathfrak{a}(\gamma_{k}))^{-1}h(\omega)$ is $\mathcal{F}$-$\mathfrak{B}(L^{1}(S))$-measurable as it is the composition of measurable functions. Consequently, $\Omega \ni \omega \mapsto \id_{\Omega_{k}}(\omega)(Id+\mathfrak{a}(\gamma_{k}))^{-1}h(\omega) $ is also $\mathcal{F}$-$\mathfrak{B}(L^{1}(S))$-measurable, which yields the measurability of $\varphi$.
\end{proof}

\begin{theorem}\label{rctheorem} One has
	\begin{align}
	\label{denseeq}
	L^{1,\infty}(\Omega;L^{1}(S)) \subseteq R(Id+A).
	\end{align}
	Consequently, the following assertions hold.
	\begin{enumerate}
		\item $R(Id+A)$ is a dense subset of $L^{1}(\Omega;L^{1}(S))$.
		\item $\A$ is m-accretive.
	\end{enumerate}
\end{theorem}
\begin{proof} Lemma \ref{denselemma} yields that (\ref{denseeq}) implies i). Moreover, one trivially has $R(Id+A)\subseteq R(Id+\A)$. Consequently, (\ref{denseeq}) implies that $R(Id+\A)$ is also dense. Moreover, as $\A$ is accretive and closed, one has that $R(Id+\A)$ is closed, cf. \cite[Proposition 2.18]{BenilanBook}. Consequently, (\ref{denseeq}) implies i) as well as ii).\\
	Now prove inclusion (\ref{denseeq}). Let $h \in L^{1,\infty}(\Omega;L^{1}(S))$.\\
	Let $f:\Omega \rightarrow L^{1}(S)$ be defined by $f(\omega):=(Id+a(g(\omega)))^{-1}h(\omega)$ for $\P$-a.e. $\omega \in \Omega$. This is well-defined by Lemma \ref{aprop}.ii).\\
	Now introduce $\hat{f}:\Omega \rightarrow L^{1}(S)$ by $\hat{f}(\omega):=a(g(\omega))f(\omega)$ for $\P$-a.e. $\omega \in \Omega$.\\
	Trivially, $h=f+\hat{f}$ by construction. Consequently the claim follows if $(f,\hat{f}) \in A$. Proving this result is divided in the following steps.
	\begin{compactenum}[(I)]
		\item $f$ as well as $\hat{f}$ are $\mathcal{F}$-$\mathfrak{B}(L^{1}(S))$-measurable.
		\item $f,~\hat{f} \in L^{1}(\Omega;L^{1}(S))$.
		\item $(f,\hat{f}) \in A$.
	\end{compactenum}
	Proof of (I). Since $\hat{f}=h-f$ it suffices to prove that $f$ is $\mathcal{F}$-$\mathfrak{B}(L^{1}(S))$-measurable.\\
	As $g \in L^{1}_{g_{1},g_{2}}(\Omega;L^{1}(S))$, there is a sequence of simple functions $(\gamma_{m})_{m \in \mathbb{N}} \subseteq L^{1}_{g_{1},g_{2}}(\Omega;L^{1}(S))$ such that $\lim \limits_{m \rightarrow \infty}\gamma_{m}(\omega)=g(\omega)$ in $L^{1}(S)$ for $\P$-a.e. $\omega \in \Omega$.\\ 
	Consequently, it follows by virtue of Lemma \ref{fancylemma} that
	\begin{align*}
	f(\omega)=(Id+a(g(\omega)))^{-1}h(\omega) = \wlim \limits_{m \rightarrow \infty} (Id+a(\gamma_{m}(\omega)))^{-1}h(\omega) \text{, in } L^{1}(S) \text{ for $\P$-a.e. } \omega \in \Omega,
	\end{align*}
	i.e. $f$ is, by Lemma \ref{measlemma}, almost surely the $L^{1}(S)$-weak limit of $\mathcal{F}$-$\mathfrak{B}(L^{1}(S))$-measurable functions and consequently it is itself $\mathcal{F}$-$\mathfrak{B}(L^{1}(S))$-measurable.\\ 
	Proof of (II). Since particularly $h \in L^{1}(\Omega;L^{1}(S))$ it suffices to prove that $f \in L^{1}(\Omega;L^{1}(S))$.\\
	The needed measurability condition has been proven in (I).\\ 
	As $(f(\omega),\hat{f}(\omega)) \in a(g(\omega))$ for $\P$-a.e. $\omega \in \Omega$ and as $a(g(\omega))$ is completely accretive, one obtains in particular 
	\begin{align*}
	\int \limits_{S}|f(\omega)| d \lambda \leq \int \limits_{S}|f(\omega)+\hat{f}(\omega)| d \lambda = \int \limits_{S} |h(\omega)|d\lambda \text{ for $\P$-a.e. } \omega \in \Omega.
	\end{align*}
	This obviously implies $\int \limits_{\Omega} ||f(\omega)||_{L^{1}(S)}d\P(\omega)\leq \int \limits_{\Omega} ||h(\omega)||_{L^{1}(S)}d\P(\omega)<\infty$, i.e. $f \in L^{1}(\Omega;L^{1}(S))$.\\ 
	Proof of (III). One has $f,\hat{f} \in L^{1}(\Omega;L^{1}(S))$ and trivially $(f(\omega),\hat{f}(\omega)) \in a(g(\omega))$ $\P$-a.e. $\omega \in \Omega$, which yields (III) by Lemma \ref{lemmaAarel}.
\end{proof}

\begin{lemma}\label{lemmadomaindense} $D(A)$ as well as $D(\A)$ are dense subsets of $(L^{1}(\Omega;L^{1}(S)),||\cdot||_{L^{1}(\Omega;L^{1}(S))})$.
\end{lemma}
\begin{proof} As $D(A) \subseteq D(\mathcal{A})$, it suffices to prove the claim for $D(A)$. Moreover, Lemma \ref{denselemma} yields that it suffices to prove that
\begin{align*} 
\tau (L^{1}(\Omega;L^{1}(S))) \subseteq \overline{D(A)}^{L^{1}(\Omega;L^{1}(S))}.
\end{align*}
Let $h \in \tau (L^{1}(\Omega;L^{1}(S))) $ and introduce $k \in (0,\infty)$, $\tilde{h} \in L^{1}(\Omega;L^{1}(S))$ such that $h=\tau_{k}(\tilde{h})$.\\
As $\A$ is m-accretive there is for each $m \in \mathbb{N}$ a uniquely determined pair of functions $(f_{m},\hat{f}_{m})\in \A$, such that
\begin{align*}
h=f_{m}+\frac{1}{m}\hat{f}_{m}.
\end{align*}
Moreover, the last equation yields, by observing that obviously $(0,0)\in\A$ and by recalling Proposition \ref{lemmacaversion} that
\begin{align*}
||f_{m}(\omega)||_{L^{\infty}(S)} \leq ||f_{m}(\omega)+\frac{1}{m}\hat{f}_{m}(\omega)||_{L^{\infty}(S)}  =||h(\omega)||_{L^{\infty}(S)}\leq k,~\forall m \in \mathbb{N}~\text{ and for $\P$-a.e. } \omega \in \Omega.
\end{align*}
Consequently, one has in particular $f_{m} \in L^{1,\infty}(\Omega;L^{1}(S))$ and hence it follows, by invoking Lemma \ref{propclosaarelation} that $(f_{m},\hat{f}_{m})\in A$ for each $m \in \mathbb{N}$.\\
Hence the claim follows if one proves that
\begin{align}
\label{densedomainlemmaproof1}
\lim \limits_{m \rightarrow \infty} f_{m}=h \text{, in } L^{1}(\Omega;L^{1}(S)).
\end{align}
Firstly, $(f_{m},\hat{f}_{m})\in A$ yields  $(f_{m}(\omega),\hat{f}_{m}(\omega))\in a(g(\omega))$ for $\P$-a.e. $\omega \in \Omega$, cf. Lemma \ref{lemmaAarel}. Moreover, $h(\omega)\in L^{\infty}(S),~h(\omega)=f_{m}(\omega)+\frac{1}{m}\hat{f}(\omega)$, i.e. $f_{m}(\omega)=(Id+\frac{1}{m}a(g(\omega)))^{-1}h(\omega)$ for all $m \in \mathbb{N}$ and $\P$-a.e. $\omega \in \Omega$. Consequently, one obtains by virtue of Lemma \ref{deterdomdenselemma} that
\begin{align*}
\lim \limits_{m \rightarrow \infty} ||f_{m}(\omega)-h(\omega)||_{L^{1}(S)}=0, \text{ for $\P$-a.e. } \omega \in \Omega.
\end{align*} 
Finally, observe that $||f_{m}(\omega)-h(\omega)||_{L^{1}(S)} \leq 2k \lambda(S)$ for all $ m \in \mathbb{N}$, and $\P$-a.e. $\omega \in \Omega$, which yields (\ref{densedomainlemmaproof1}), by virtue of dominated convergence.
\end{proof}

\begin{theorem} Let $u \in L^{1}(\Omega;L^{1}(S))$. The evolution equation (\ref{eveqintro2}) has, for any $u \in L^{1}(\Omega;L^{1}(S))$, a uniquely determined mild solution.
\end{theorem}
\begin{proof} The claim follows from Theorem \ref{lemmamsg}, Theorem \ref{rctheorem} and Lemma \ref{lemmadomaindense}.
\end{proof}

\section{Strong solutions of the randomized weighted $p$-Laplacian evolution equation with Neumann boundary conditions}
\label{section_ss} 
Throughout everything which follows $(T(t))_{t \geq 0}$ denotes the semigroup of mild solutions of (\ref{eveqintro2}), i.e. \linebreak$T(t):L^{1}(\Omega;L^{1}(S)) \rightarrow L^{1}(\Omega;L^{1}(S))$ is such that $T(\cdot)u \in C([0,\infty);L^{1}(\Omega;L^{1}(S)))$ is the uniquely determined mild solution of (\ref{eveqintro2}), fulfilling $T(0)u=u$, for any $u \in L^{1}(\Omega;L^{1}(S))$.\\
Now let $\gamma \in L^{1}_{g_{1},g_{2}}(S)$. Then, throughout everything which follows, $(T(t,\gamma))_{t \geq 0}$ denotes the semigroup of mild solutions of
\begin{align*}
0 \in V^{\prime}(t)+\mathfrak{a}(\gamma) V(t),~ \text{for a.e. } t \in (0,\infty),~ V(0)=v \in L^{1}(S). 
\end{align*}
The existence of mild solutions is clear, because $\mathfrak{a}(\gamma)$ is m-accretive and has dense domain.\\

The purpose of this section is to establish that $t \mapsto T(t)u$ is, for any $u \in L^{1}(\Omega;L^{1}(S))$, not only a mild, but also a strong solution of (\ref{eveqintro2}).\\
The following theorem serves as a guideline for this section:

\begin{theorem}\label{lemmassg} Let $(X,||\cdot||)$ be a Banach space and let $B:D(B)\rightarrow 2^{X}$ be closed and m-accretive. Moreover, assume that $D(B)$ is dense in $X$ and let $U \in C([0,\infty);X)$ be the mild solution of
\begin{align}
\label{eveqgeneral}
0 \in U^{\prime}(t)+BU(t),~\text{for a.e. } t \in (0,\infty),~ U(0)=u \in X.
\end{align}
If $U \in W^{1,1}_{\text{Loc}}((0,\infty);X)$, one has that $U$ is also the uniquely determined strong solution of (\ref{eveqgeneral}).
\end{theorem}
\begin{proof} Firstly, note that any strong solution of (\ref{eveqgeneral}) is also a mild solution. Consequently, Theorem \ref{lemmamsg} yields the uniqueness. Moreover, $U \in W^{1,1}_{\text{Loc}}((0,\infty);X)$ already implies that $U$ is a strong solution of (\ref{eveqgeneral}), cf. \cite[Prop. 7.9]{BenilanBook}. 
\end{proof}

\begin{remark}\label{remarksemicontracexpf} Let $\gamma \in L^{1}_{g_{1},g_{2}}(S)$. Then one has, by using the nonlinear semigroup theory, that $T(t)$ as well as $T(t,\gamma)$ are contraction semigroups fulfilling the exponential formula, i.e. one has
	\begin{enumerate}
		\item $T(t_{1}+t_{2})u=T(t_{1})T(t_{2})u$ and $T(t_{1}+t_{2},\gamma)v=T(t_{1},\gamma)T(t_{2},\gamma)v$, for all $t_{1},t_{2} \in [0,\infty)$,\linebreak $u \in L^{1}(\Omega;L^{1}(S))$ and $v \in L^{1}(S)$  (cf. \cite[Theorem 1.10]{BenilanBook}),
		\item $||T(t)u_{1}-T(t)u_{2}||_{L^{1}(\Omega;L^{1}(S))} \leq  ||u_{1}-u_{2}||_{L^{1}(\Omega;L^{1}(S))}$, for all $u_{1},u_{2} \in L^{1}(\Omega;L^{1}(S))$ and \linebreak$||T(t,\gamma)v_{1}-T(t,\gamma)v_{2}||_{L^{1}(S)} \leq  ||v_{1}-v_{2}||_{L^{1}(S)}$ for all $v_{1},v_{2} \in L^{1}(S)$ (cf. \cite[Theorem 3.10]{BenilanBook}),
		\item $\lim \limits_{m \rightarrow \infty} \left(Id+\frac{t}{m}\A\right)^{-m}u = T(t)u$ for all $u \in L^{1}(\Omega;L^{1}(S))$ and $\lim \limits_{m \rightarrow \infty} \left(Id+\frac{t}{m}\mathfrak{a}(\gamma)\right)^{-m}v= T(t,\gamma)v$ for all $v \in L^{1}(S)$, where the limits are understood in $L^{1}(\Omega;L^{1}(S))$ and in $L^{1}(S)$, resp. (cf. \cite[Theorem 4.2]{BenilanBook}).
	\end{enumerate}
\end{remark}

\begin{theorem}\label{deterministicidentityprop} Let $u \in L^{1}(\Omega;L^{1}(S))$ and $t \in [0,\infty)$. Then one has
\begin{align*}
P\left(\left\{ \omega \in \Omega:~(T(t)u)(\omega)=T(t,g(\omega))u(\omega)\right\}\right)=1.
\end{align*} 
\end{theorem}
\begin{proof} Let $u \in L^{1}(\Omega;L^{1}(S))$ and $\tilde{t} \in (0,\infty)$.\\ 
Firstly, it will be proven inductively that
\begin{align}
\label{inductionhypo}
((Id+\tilde{t}\A)^{-m}u)(\omega)=(Id+\tilde{t}\mathfrak{a}(g(\omega)))^{-m}(u(\omega)),~\forall m \in \mathbb{N} \text{ and for $\P$-a.e. } \omega \in \Omega.
\end{align}
So let $m=1$ and introduce $f:=(Id+\tilde{t}\A)^{-1}u$. Consequently, there is an $\hat{f} \in \A f$ such that $f+\tilde{t}\hat{f}=u$.\\
As $(f,\hat{f}) \in \A$, one has $(f(\omega),\hat{f}(\omega)) \in \mathfrak{a}(g(\omega))$ for $\P$-a.e. $\omega \in \Omega$, see Lemma \ref{lemmaAclosarel}.\\
Since obviously $f(\omega)+\tilde{t}\hat{f}(\omega)=u(\omega)$ for $\P$-a.e. $\omega \in \Omega$ one obtains that $f(\omega)=(Id+\tilde{t}\mathfrak{a}(g(\omega)))^{-1}(u(\omega))$ for $\P$-a.e. $\omega \in \Omega$ and consequently
\begin{align*}
((Id+\tilde{t}\A)^{-1}u)(\omega)=f(\omega)=(Id+\tilde{t}\mathfrak{a}(g(\omega)))^{-1}(u(\omega)) \text{ for $\P$-a.e. } \omega \in \Omega,
\end{align*}
i.e. (\ref{inductionhypo}) is proven for $m=1$. The proof of the induction step works analogously and will be skipped.\\

Now let $t \in [0,\infty)$ be given and choose $\tilde{t}:=\frac{t}{m}$ in (\ref{inductionhypo}). Then one gets
\begin{align}
\label{deterministicrelproof1}
\left(\left(Id+\frac{t}{m}\A\right)^{-m}u\right)(\omega)=\left(Id+\frac{t}{m}\mathfrak{a}(g(\omega))\right)^{-m}u(\omega),~\forall m \in \mathbb{N} \text{ and $\P$-a.e. } \omega \in \Omega.
\end{align}
Moreover, the exponential formula, i.e. Remark \ref{remarksemicontracexpf}.iii), yields, by passing to a subsequence if necessary, that
\begin{align*} 
\lim \limits_{m \rightarrow \infty} \left(\left(Id+\frac{t}{m}\A\right)^{-m}u\right)(\omega) = (T(t)u)(\omega),~\text{for $\P$-a.e. } \omega \in \Omega \text{ in } L^{1}(S).
\end{align*}
Analogously, one also has by virtue of the exponential formula that
\begin{align*}
\lim \limits_{m \rightarrow \infty} \left(Id+\frac{t}{m}\mathfrak{a}(g(\omega))\right)^{-m}u(\omega)= T(t,g(\omega))u(\omega),~\text{for $\P$-a.e. }\omega \in \Omega \text{, in } L^{1}(S),
\end{align*}
which yields the claim.
\end{proof}

Even though one has $T(\cdot,\gamma)u \in W^{1,1}_{\text{Loc}}((0,\infty);L^{1}(S))$, cf. \cite[Theorem 3.7]{main} and in addition \linebreak$T(t,g(\omega))u(\omega)=(T(t)u)(\omega)$ for $\P$-a.e. $ \omega \in \Omega$, it is not trivial that $T(\cdot)u$ is also differentiable almost everywhere and locally Lipschitz continuous. Particularly one needs a more detailed version of some parts of the proof of \cite[Theorem 3.7]{main}. It is only stated in \cite{main} that $T(\cdot,\gamma)v$ is locally Lipschitz continuous and differentiable a.e. Particularly, the Lipschitz constant of $T(\cdot,\gamma)v$ is not given and moreover it is not explicitly proven that $T(\cdot,\gamma)u$ is right differentiable everywhere on $(0,\infty)$. The Lipschitz constant as well as the right differentiability of $T(\cdot,\gamma)u$ are needed to prove that $T(\cdot)u \in W^{1,1}_{\text{Loc}}((0,\infty);L^{1}(\Omega;L^{1}(S)))$.

\begin{lemma}\label{lemmadetdifflips} Let $v \in L^{1}(S)$ and $\gamma \in L^{1}_{g_{1},g_{2}}(S)$. The mapping $(0,\infty) \ni t \mapsto T(t,\gamma)v$ is right differentiable, i.e. there is $T^{\prime}_{r}(t,\gamma)v \in L^{1}(S)$ such that
\begin{align}
\label{lemmadetdifflipseq}
\lim \limits_{h \searrow 0} \frac{T(t+h,\gamma)v-T(t,\gamma)v}{h}=T^{\prime}_{r}(t,\gamma)v \text{, in } L^{1}(S),~\forall t \in (0,\infty).
\end{align}
In addition, let $\varepsilon > 0$. Then the mapping $[\varepsilon,\infty) \ni t \mapsto T(t,\gamma)v$ is Lipschitz continuous; more precisely one has
\begin{align*}
||T(t+h,\gamma)v-T(t,\gamma)v||_{L^{1}(S)} \leq h ||T^{\prime}_{r}(t,\gamma)v||_{L^{1}(S)} \leq h \frac{2}{|p-2|\varepsilon}||v||_{L^{1}(S)},~\forall t \geq \varepsilon,~h>0.
\end{align*} 
\end{lemma}
\begin{proof} Let $\gamma \in L^{1}_{g_{1},g_{2}}(S)$ and $v \in L^{1}(S)$.\\
Throughout this proof let $\mathfrak{a}(\gamma)^{\circ}: D(\mathfrak{a}(\gamma)) \rightarrow L^{1}(S)$ be such that $\mathfrak{a}(\gamma)^{\circ}w$ is, for any $w \in D(\mathfrak{a}(\gamma))$, the uniquely determined element fulfilling
\begin{align*}
\mathfrak{a}(\gamma)^{\circ}w<< \hat{w},~\forall \hat{w} \in \mathfrak{a}(\gamma)w.
\end{align*}
As $\mathfrak{a}(\gamma)$ is m-accretive and completely accretive, such an element does indeed exist and is unique, cf. \cite[Proposition 3.7]{cao}.\\

As $\mathfrak{a}(\gamma)$ is m-accretive and completely accretive, \cite[Theorem 4.2]{cao} yields that
\begin{align*}
\lim \limits_{h \searrow 0 } \frac{T(h,\gamma)w-w}{h} = - \mathfrak{a}(\gamma)^{\circ}w,~\text{in } L^{1}(S),~\forall w \in D(\mathfrak{a}(\gamma)).
\end{align*}
Moreover, as $\mathfrak{a}(\gamma)$ is also positively homogeneous of degree $p-1$ and since $\overline{D(\mathfrak{a}(\gamma))}=L^{1}(S)$, it follows by virtue of \cite[Theorem 4.4]{cao}  that $T(t,\gamma)v \in D(\mathfrak{a}(\gamma))$ for every $t \in (0,\infty)$.\\
Now introduce $T^{\prime}_{r}(t,\gamma)v:= -\mathfrak{a}(\gamma)^{\circ}T(t,\gamma)v$ for any $t \in (0,\infty)$. Then one has for any $t \in (0,\infty)$ that
\begin{align*}
\lim \limits_{h \searrow 0} \frac{T(t+h,\gamma)v-T(t,\gamma)v}{h} = \lim \limits_{h \searrow 0} \frac{T(h,\gamma)T(t,\gamma)v-T(t,\gamma)v}{h} = -\mathfrak{a}(\gamma)^{\circ}T(t,\gamma)v = T^{\prime}_{r}(t,\gamma)v,~\text{in } L^{1}(S)
\end{align*} 
which yields the first assertion.\\
Now let $\varepsilon>0$, $t \in [\varepsilon,\infty)$ and $h \in (0,\infty)$. As $T(t,\gamma)v \in D(\mathfrak{a}(\gamma))$, one infers by \cite[Theorem 4.2]{cao} and \cite[Theorem 4.4]{cao} that
\begin{align*}
\left|\left| \frac{T(t+h,\gamma)v-T(t,\gamma)v}{h} \right|\right|_{L^{1}(S)} \leq || T^{\prime}_{r}(t,\gamma)v||_{L^{1}(S)} \leq  \frac{2}{|p-2|t}||v||_{L^{1}(S)} \leq \frac{2}{|p-2|\varepsilon}||v||_{L^{1}(S)}
\end{align*}
which yields the claim.
\end{proof}

\begin{remark} Throughout everything which follows $T^{\prime}_{r}(t,\gamma)v$ denotes, for any $\gamma \in L^{1}_{g_{1},g_{2}}(S)$, $t>0$ and $v \in L^{1}(S)$, the element introduced in (\ref{lemmadetdifflipseq}).\\
Moreover, introduce $T^{\prime}_{r}(t)u: \Omega \rightarrow L^{1}(S)$, for any $u \in L^{1}(\Omega;L^{1}(S))$ and $t>0$, by 
\begin{align*}
(T^{\prime}_{r}(t)u)(\omega):=T^{\prime}_{r}(t,g(\omega))u(\omega).
\end{align*}
\end{remark}

\begin{lemma}\label{lemmalcrd} Let $u \in L^{1}(\Omega;L^{1}(S))$. The mapping $(0,\infty) \ni t \mapsto T(t)u$ is right differentiable; more precisely, one has
	\begin{align*}
	\lim \limits_{h \searrow 0} \frac{T(t+h)u-T(t)u}{h}=T^{\prime}_{r}(t)u \text{, in } L^{1}(\Omega;L^{1}(S)),~\forall t \in (0,\infty).
	\end{align*}
In addition, let $\varepsilon > 0$. Then the mapping $[\varepsilon,\infty) \ni t \mapsto T(t)u$ is Lipschitz continuous, i.e.
\begin{align}
\label{randlipseq}
||T(t+h)u-T(t)u||_{L^{1}(\Omega;L^{1}(S))} \leq h ||T^{\prime}_{r}(t)u||_{L^{1}(\Omega;L^{1}(S))} \leq  h \frac{2}{|p-2|\varepsilon}||u||_{L^{1}(\Omega;L^{1}(S))}
\end{align}
for all $t \geq \varepsilon$ and $h>0$.
\end{lemma}
\begin{proof} Let $u \in L^{1}(\Omega;L^{1}(S))$, $t \in (0,\infty)$ and $(h_{m})_{m \in \mathbb{N}} \subseteq (0,\infty)$ be a sequence fulfilling $\lim \limits_{m \rightarrow \infty}h_{m}=0$.\\
One has, by virtue of Theorem \ref{deterministicidentityprop} and Lemma \ref{lemmadetdifflips} that
\begin{align*}
\lim \limits_{m \rightarrow \infty} \left|\left| (T^{\prime}_{r}(t)u)(\omega)- \frac{(T(t+h_{m})u)(\omega)-(T(t)u)(\omega)}{h_{m}}\right|\right|_{L^{1}(S)}=0,~\text{for $\P$-a.e. } \omega \in \Omega.
\end{align*}
Consequently, $T^{\prime}_{r}(t)u$ is $\F-\mathfrak{B}(L^{1}(S))$-measurable because it is a.s. the $L^{1}(S)$-limit of $\F$-$\mathfrak{B}(L^{1}(S))$-measurable functions.\\
Moreover, one has by applying Theorem \ref{deterministicidentityprop} and Lemma \ref{lemmadetdifflips} again that
\begin{align*}
 \left|\left| (T^{\prime}_{r}(t)u)(\omega)- \frac{(T(t+h_{m})u)(\omega)-(T(t)u)(\omega)}{h_{m}}\right|\right|_{L^{1}(S)} \leq 2 ||T^{\prime}_{r}(t,g(\omega))u(\omega)||_{L^{1}(S)} \leq \frac{4}{|p-2|t}||u(\omega)||_{L^{1}(S)},
\end{align*}
for $\P$-a.e. $\omega \in \Omega$ and all $m \in \mathbb{N}$. Consequently, as $u \in L^{1}(\Omega;L^{1}(S))$ one infers from Lebesgue's theorem that
\begin{align*}
\lim \limits_{m \rightarrow \infty}  \frac{T(t+h_{m})u-T(t)u}{h_{m}} = T^{\prime}_{r}(t)u,~\text{in } L^{1}(\Omega;L^{1}(S)).
\end{align*} 
Now let $\varepsilon>0$, $h>0$ and assume $t\geq \varepsilon$. Then one has by Lemma \ref{lemmadetdifflips} and Theorem \ref{deterministicidentityprop} that
\begin{align*}
||T(t+h)u)(\omega)-(T(t)u)(\omega)||_{L^{1}(S)}  \leq h ||(T^{\prime}_{r}(t)u)(\omega)||_{L^{1}(S)} \leq h \frac{2}{|p-2|\varepsilon}||u(\omega)||_{L^{1}(S)},~\text{ for $\P$-a.e. } \omega \in \Omega,
\end{align*}
which obviously implies (\ref{randlipseq}).
\end{proof}

\begin{remark} Let $(X,||\cdot||)$ denote a separable Banach space and let $f: \Omega \rightarrow X$ be any mapping. So far it has been used that the notions of $f$ being $\F-\mathfrak{B}(X)-$measurable, strongly measurable and weakly measurable are equivalent,  cf. \cite[Theorem 5.8 and 5.9]{meas}.\\
Now it will be necessary to consider measurability on the not necessarily separable Banach space\linebreak $L^{1}(\Omega;L^{1}(S))$.
Consequently, if $f: [a,b] \rightarrow L^{1}(\Omega;L^{1}(S))$, one needs to choose one of the measurability concepts.\\ 
Hereby, the common convention of choosing strong measurability as the appropriate notion of measurability will be followed. Moreover, introduce
\begin{align*}
L^{1}([a,b];L^{1}(\Omega;L^{1}(S))):=\{f:[a,b]\rightarrow L^{1}(\Omega;L^{1}(S))|~f \text{ is strongly meas., } \int \limits_{[a,b]} ||f(t)||_{L^{1}(\Omega;L^{1}(S))}dt<\infty\},
\end{align*}
for any compact interval $[a,b] \subseteq [0,\infty)$.\\
For an introduction to strong measurability, resp. integrability of functions $f:[a,b] \rightarrow X$, for not necessarily separable Banach spaces $(X,||\cdot||)$, see \cite{greenbook}. Particularly, all results concerning integrability, resp. differentiability, of vector-valued functions needed in the following proof, can be found in \cite[Section 1.1 and 1.2]{greenbook}.
\end{remark}

\begin{theorem}\label{theoremmainresultex} Let $u \in L^{1}(\Omega;L^{1}(S))$. Then one has $T(\cdot)u \in W^{1,1}_{\text{Loc}}((0,\infty);L^{1}(\Omega;L^{1}(S)))$ and therefore
\begin{align*}
0 \in T^{\prime}(t)u+\A T(t)u,~ \text{for a.e. } t \in (0,\infty),~ T(0)u=u,
\end{align*}
i.e. $T(\cdot)u$ is not only the mild, but also the uniquely determined strong solution of (\ref{eveqintro2}). 
\end{theorem}
\begin{proof}  Recalling Theorem \ref{lemmassg} yields that it suffices to prove $T(\cdot)u \in W^{1,1}_{\text{Loc}}((0,\infty);L^{1}(\Omega;L^{1}(S)))$. Moreover, by Lemma \ref{lemmalcrd} one has that this holds if $(0,\infty) \ni t \mapsto T(t)u$ is differentiable almost everywhere. In addition, it is clear that this follows if $(\frac{1}{m},m) \ni t \mapsto T(t)u$ is differentiable almost everywhere for every $m \in \mathbb{N}$.\\
Let $m \in \mathbb{N}$ be arbitrary but fixed. Moreover, let $\zeta :[\frac{1}{m},m]\rightarrow L^{1}(\Omega;L^{1}(S))$ be defined by $\zeta(t):=T(t)u$ for all $t \in [\frac{1}{m},m]$. Then $\zeta$ is, as it is continuous, a fortiori strongly measurable. Moreover, the continuity also yields that
\begin{align*}
\int \limits_{[\frac{1}{m},m]} ||\zeta(t)||_{L^{1}(\Omega;L^{1}(S))} dt < \infty
\end{align*}
and therefore $\zeta \in L^{1}([\frac{1}{m},m];L^{1}(\Omega;L^{1}(S)))$.\\
Now introduce $\zeta^{\prime}_{r}:[\frac{1}{m},m]\rightarrow L^{1}(\Omega;L^{1}(S))$, by $\zeta^{\prime}_{r}(t):=T^{\prime}_{r}(t)u$. Then one has
\begin{align*}
\lim \limits_{h \searrow 0 } \left|\left|\frac{\zeta(t+h)-\zeta(t)}{h}-\zeta^{\prime}_{r}(t)\right|\right|_{L^{1}(\Omega;L^{1}(S))}=0,
\end{align*}
for all $t \in [\frac{1}{m},m)$. Moreover, it follows by virtue of Lemma \ref{lemmalcrd} that
\begin{align*}
\left|\left|\frac{\zeta(t+h)-\zeta(t)}{h}-\zeta^{\prime}_{r}(t)\right|\right|_{L^{1}(\Omega;L^{1}(S))} \leq \frac{4m}{|p-2|}||u||_{L^{1}(\Omega;L^{1}(S))},~\forall t \in [\frac{1}{m},m],~h>0.
\end{align*}
Consequently, dominated convergence yields $\lim \limits_{h \searrow 0} \frac{\zeta(\cdot+h)-\zeta(\cdot)}{h} = \zeta^{\prime}_{r}(\cdot)$ in $L^{1}([\frac{1}{m},m];L^{1}(\Omega;L^{1}(S)))$ and therefore particularly that $\zeta^{\prime}_{r}(\cdot) \in L^{1}([\frac{1}{m},m];L^{1}(\Omega;L^{1}(S)))$.\\
Now introduce $\zeta_{\ast}:[\frac{1}{m},m]\rightarrow L^{1}(\Omega;L^{1}(S))$, by
\begin{align*}
\zeta_{\ast}(t):= \int \limits_{[\frac{1}{m},t]} \zeta^{\prime}_{r}(z)dz+\zeta\left(\frac{1}{m}\right),~\forall t \in [\frac{1}{m},m].
\end{align*}
Then the fundamental theorem of calculus (for Bochner integrals) yields that $\zeta_{\ast}$ is differentiable almost everywhere and that $\zeta_{\ast}^{\prime}(t)=\zeta^{\prime}_{r}(t)$ for a.e. $t \in [\frac{1}{m},m]$.\\
Consequently, the claim follows if $\zeta(t)=\zeta_{\ast}(t)$ for every $t \in [\frac{1}{m},m]$.\\
To prove this, introduce $\Gamma:[\frac{1}{m},m] \rightarrow \mathbb{R}$ by
\begin{align*}
\Gamma(t):= ||\zeta(t)-\zeta_{\ast}(t)||_{L^{1}(\Omega;L^{1}(S))},~\forall t \in [\frac{1}{m},m].
\end{align*}
Firstly, note that obviously $\Gamma(\frac{1}{m})=0$. Moreover, one has
\begin{eqnarray*}
& & ~
	\lim \limits_{h \searrow  0}   \left|\frac{\Gamma(t+h)-\Gamma(t)}{h}\right| \\
& \leq  & ~ \lim \limits_{h \searrow  0} \left(\left| \left| \frac{\zeta(t+h)-\zeta(t)}{h}+\zeta^{\prime}_{r}(t) \right|\right|_{L^{1}(\Omega;L^{1}(S))} + \left| \left| \frac{-\zeta_{\ast}(t+h)+\zeta_{\ast}(t)}{h}-\zeta^{\prime}_{r}(t) \right|\right|_{L^{1}(\Omega;L^{1}(S))}\right) \\ 
& = & ~ 0,
\end{eqnarray*}
for almost every $t \in [\frac{1}{m},m)$, i.e.  $\Gamma$ is almost everywhere right differentiable and the right derivative is equal to zero.\\
In addition, one has by invoking Lemma \ref{lemmalcrd} that
\begin{align*}
|\Gamma(t+h)-\Gamma(t)| \leq h \frac{2m}{|p-2|}||u||_{L^{1}(\Omega;L^{1}(S))}+\int \limits_{[t,t+h]} ||\zeta^{\prime}_{r}(z)||_{L^{1}(\Omega;L^{1}(S))}dz \leq h \frac{4m}{|p-2|}||u||_{L^{1}(\Omega;L^{1}(S))}
\end{align*}
for all $t \in [\frac{1}{m},m]$ and $0 < h \leq m-t$.\\
Conclusively, the last estimate yields that $\Gamma$ is Lipschitz continuous, which implies, as $\mathbb{R}$ has the Radon-Nikodym property, that it is differentiable almost everywhere. Since the right derivate of $\Gamma$ is zero almost everywhere, the almost everywhere derivative is also zero a.e. Finally, the Lipschitz continuity of $\Gamma$ yields that $\Gamma$ is constant, and hence $\Gamma(t)=0$ for all $t \in [\frac{1}{m},m]$.
\end{proof}

\begin{remark}\label{remarkccsg} Let $\gamma \in L^{1}_{g_{1},g_{2}}(S)$. Then $T(\cdot,\gamma)$ is not only a contraction semigroup, but even a complete contraction semigroup, i.e. one has
\begin{align*}
T(t,\gamma)v_{1}-T(t,\gamma)v_{2} << v_{1}-v_{2}
\end{align*}
for every $v_{1},v_{2} \in L^{1}(S)$ and $t \geq 0$, cf. \cite[Proposition 4.1]{cao}.\\
This also yields that $T(t,\gamma)v<<v$ for all $t \geq 0$ and $v \in L^{1}(S)$, because one immediately verifies that $T(t,\gamma)0=0$. (Due to the non-linearity, this is actually not true for arbitrary semigroups.) 
\end{remark}

\begin{proposition}\label{strongsolclosprop} Let $u,u_{1},u_{2} \in L^{1}(\Omega;L^{1}(S))$, $q \in [1,\infty]$ and $t \in [0,\infty)$. Moreover, assume that $u,u_{1},u_{2} \in L^{q}(S)$ a.s. Then the following assertions hold.
\begin{enumerate}
\item $\P \left(\big\{\omega \in \Omega:~||(T(t)u_{1}) (\omega)-(T(t)u_{2}) (\omega)||_{L^{q}(S)}\leq ||u_{1}(\omega)-u_{2}(\omega)||_{L^{q}(S)} \big\} \right) =1$,
\item $\P \left(\big\{\omega \in \Omega:~||(T(t)u) (\omega)||_{L^{q}(S)}\leq ||u(\omega)||_{L^{q}(S)} \big\} \right) =1$.
\end{enumerate} 
\end{proposition}
\begin{proof} Follows trivially from Remark \ref{remarkccsg}, Theorem \ref{deterministicidentityprop} and Remark \ref{remarkca}.
\end{proof}

\begin{theorem}\label{mainresult} Let $u \in L^{1,\infty}(\Omega;L^{1}(S))$. Then one has
\begin{align}
\label{mainresulteq}
-T^{\prime}(t)u=AT(t)u,~\text{for a.e. } t \in (0,\infty).
\end{align}
\end{theorem}
\begin{proof} Let $u \in L^{1,\infty}(\Omega;L^{1}(S))$. Theorem \ref{theoremmainresultex} yields that $(T(t)u,-T^{\prime}(t)u) \in \A$ for a.e. $t \in (0,\infty)$. Consequently, it follows by virtue of Lemma \ref{propclosaarelation} that it suffices to prove $\P( T(t)u \in L^{\infty}(S) ) =1,~\text{for a.e. }$ \linebreak$ t \in (0,\infty)$. But this is a trivial consequence of Proposition \ref{strongsolclosprop}.ii).
\end{proof}

Theorem \ref{mainresult} finishes the discussion on existence and uniqueness results. The remaining part of this paper is devoted to determine the asymptotic behavior of $T(t)u$.

\section{Stability results for the solution of the randomized weighted p-Laplace evolution equation with Neumann boundary conditions}
\label{section_gar} 
This section opens the investigation on the asymptotic results of $T(t)$. The asymptotic properties of $T(t,\gamma)$ have been discussed in \cite{ich}. Due to Theorem \ref{deterministicidentityprop} some of these results can be easily transfered to the current setting.\\

For any $v \in L^{1}(S)$, let $\overline{(v)}$ denote its average, i.e. $\overline{(v)}:= \frac{1}{\lambda(S)} \int \limits_{S} v d \lambda$. By slightly abusing notation, the real number $\overline{(v)}$ and the function $\varphi : S \rightarrow \mathbb{R}$ which is constantly equal to $\overline{(v)}$, is also denoted by $\overline{(v)}$.\\
Moreover, for $u: \Omega \rightarrow L^{1}(S)$ introduce $\overline{(u)}(\omega):=\overline{(u(\omega))}$.\\ 
Finally, let $C_{S,q}$ denote the Poincar\'{e} constant of $S$ in $L^{q}(S)$, i.e. $C_{S,q} \in (0,\infty)$ is the smallest constant such that
\begin{align*}
||v-\overline{(v)}||_{L^{q}(S)} \leq C_{S,q} ||\nabla v||_{L^{q}(S;\mathbb{R}^{n})},~\forall v \in W^{1,q}(S),
\end{align*}
where $q \in [1,\infty)$. The Poincar\'{e} inequality ensures that $C_{S,q}$ exists.\\ 
Moreover, introduce the real-valued random variable $\Delta_{u}:\Omega \rightarrow [0,\infty)$ by 
\begin{align*}
\Delta_{u}(\omega):=||u(\omega)-\overline{(u(\omega))}||^{2}_{L^{2}(S)} 
\end{align*}
for any $u \in L^{1}(\Omega;L^{1}(S))$, with $\P(u \in L^{2}(S))=1$.

\begin{lemma}\label{lemmaichpropasymp} Let $\gamma \in L^{1}_{g_{1},g_{2}}(S)$. Then one has $\overline{( T(t,\gamma)v)}=\overline{(v)}$ for any $t \in [0,\infty)$ and $v \in L^{1}(S)$. In addition, 
\begin{align*}
||T(t,\gamma)v-\overline{(v)}||_{L^{1}(S)} \leq C_{S,1} \left( \int \limits_{S} \gamma^{\frac{1}{1-p}} d \lambda \right)^{\frac{p-1}{p}} \left( \frac{2}{|p-2|} \right)^{\frac{1}{p}} ||v-\overline{(v)}_{S}||^{\frac{2}{p}}_{L^{2}(S)} \left(\frac{1}{t}\right)^{\frac{1}{p}}
\end{align*} 
for any $t \in (0,\infty)$ and $v \in L^{2}(S)$.
\end{lemma}
\begin{proof} See \cite[Lemma 3.4 and Corollary 4.8]{ich}.
\end{proof} 

\begin{corollary}\label{boundL1cor} Let $t \in [0,\infty)$ and $u \in L^{1}(\Omega;L^{1}(S))$. Then $\overline{((T(t)u)(\omega))}=\overline{(u(\omega))}$ for $\P$-a.e. $\omega \in \Omega$. If in addition $\P(u \in L^{2}(S))=1$ and $t \neq 0$, then one has
\begin{align*}
||(T(t)u)(\omega)-\overline{(u(\omega))}||_{L^{1}(S)} \leq C_{S,1} \lambda(S)^{\frac{p-1}{p}}\left( \frac{2}{g_{1}|p-2|} \right)^{\frac{1}{p}} \Delta_{u}(\omega)^{\frac{1}{p}}  \left(\frac{1}{t}\right)^{\frac{1}{p}}.
\end{align*}
\end{corollary}
\begin{proof} Follows by combining Lemma \ref{lemmaichpropasymp} and Theorem \ref{deterministicidentityprop}.
\end{proof}

\begin{remark} Thanks to the last corollary, it is, for sufficiently integrable $u$, straightforward to derive upper bounds for $||T(t)u-\overline{(u)}||_{L^{q}(\Omega;L^{1}(S))}$, where $q \in [1,\infty)$.\\ 
What is not that easy is to establish is that $T(t)u-\overline{(u)}$ converges in $L^{q}(\Omega;L^{q}(S))$ to zero, if one only requires $u \in L^{q}(\Omega;L^{q}(S))$. 
\end{remark}

\begin{theorem} Let $q \in [1,\infty)$ and $u \in L^{q}(\Omega;L^{q}(S))$. Then one has
\begin{align*}
\lim \limits_{t \rightarrow \infty} T(t)u=\overline{(u)},\text{ in } L^{q}(\Omega;L^{q}(S)).
\end{align*} 
\end{theorem}
\begin{proof} Firstly, the case $q=1$ is considered. Let $u \in L^{1}(\Omega;L^{1}(S))$. Moreover, let $\varepsilon>0$ be arbitrary but fixed and introduce $\tilde{u} \in \tau (L^{1}(\Omega;L^{1}(S)))$, such that
\begin{align*}
||u-\tilde{u}||_{L^{1}(\Omega;L^{1}(S))} < \frac{\varepsilon}{3} \text{ and } \mathbb{E}||\overline{(\tilde{u})}-\overline{(u)}||_{L^{1}(S)} <  \frac{\varepsilon}{3}.
\end{align*}
Lemma \ref{denselemma} ensures the existence of $\tilde{u}$. The last inequality, together with Proposition \ref{strongsolclosprop}.ii), yields that 
\begin{align*}
||T(t)u-T(t)\tilde{u}||_{L^{1}(\Omega;L^{1}(S))} < \frac{\varepsilon}{3}.
\end{align*} 
Now let $k \in (0,\infty)$ be such that $||\tilde{u}(\omega)||_{L^{\infty}(S)} \leq k$ for $\P$-a.e. $\omega \in \Omega$. Then Corollary \ref{boundL1cor} implies that
\begin{align*}
||T(t)\tilde{u}-\overline{(\tilde{u})}||_{L^{1}(\Omega;L^{1}(S))} \leq C_{S,1} \lambda(S)^{\frac{p-1}{p}}\left( \frac{2}{g_{1}|p-2|} \right)^{\frac{1}{p}} \left(2k\sqrt{\lambda(S)}\right)^{\frac{2}{p}}\left(\frac{1}{t}\right)^{\frac{1}{p}},~\forall t >0.
\end{align*}
The last estimate yields the existence of a $t_{0} \in (0,\infty)$ such that
\begin{align*} 
||T(t)\tilde{u}-\overline{(\tilde{u})}||_{L^{1}(\Omega;L^{1}(S))} < \frac{\varepsilon}{3},~\forall t \geq t_{0}. 
\end{align*}
Hence, one has for all $t \geq t_{0}$ that
\begin{align*}
||T(t)u-\overline{(u)}||_{L^{1}(\Omega;L^{1}(S))} = ||T(t)u-T(t)\tilde{u}+T(t)\tilde{u}-\overline{(\tilde{u})}+\overline{(\tilde{u})}-\overline{(u)}||_{L^{1}(\Omega;L^{1}(S))} < \varepsilon,
\end{align*}
i.e.
\begin{align}
\label{convproof7}
\lim \limits_{t \rightarrow \infty} ||T(t)u-\overline{(u)}||_{L^{1}(\Omega;L^{1}(S))}=0,~\forall u \in L^{1}(\Omega;L^{1}(S)).
\end{align}
Now let $q \in (1,\infty)$ and $u \in L^{q}(\Omega;L^{q}(S))$.\\ 
Let $(t_{m})_{m \in \mathbb{N}} \subseteq [0,\infty)$ be such that $\lim \limits_{m \rightarrow \infty}t_{m}=\infty$. 
Moreover, let $k \in (0,\infty)$. Then (\ref{convproof7}) yields, by passing to a subsequence if necessary, that
\begin{align*}
\lim \limits_{m \rightarrow \infty} (T(t_{m})\tau_{k}(u)) (\omega) = \overline{(\tau_{k}(u))}(\omega) \text{, for $\P$-a.e. } \omega \in \Omega \text{, in } L^{1}(S).
\end{align*}
Moreover, by invoking Proposition \ref{strongsolclosprop}, one has
\begin{align*}
 ||(T(t_{m})\tau_{k}(u)) (\omega)||_{L^{\infty}(S)} \leq k,~\text{$\P$-a.e. } \omega \in \Omega,~\forall m \in \mathbb{N}.
\end{align*}
Now fix $\omega \in \Omega$ such that the last two assertions holds. Then it follows, by passing to a subsequence if necessary, that $\lim \limits_{m \rightarrow \infty} (T(t_{m})\tau_{k}(u)) (\omega) = \overline{(\tau_{k}(u))}(\omega)$ a.e. on $S$ and consequently 
\begin{align}
\label{convproof1}
\lim \limits_{m \rightarrow \infty} \left|(T(t_{m})\tau_{k}(u)) (\omega) - \overline{(\tau_{k}(u))}(\omega)\right|^{q}=0 \text{ a.e. on } S.
\end{align}
Moreover, as $||T(t_{m})\tau_{k}(u)(\omega)||_{L^{\infty}(S)} \leq  ||\overline{(\tau_{k}(u))}(\omega)||_{L^{\infty}(S)}\leq k$, one has by (\ref{convproof1}) and by virtue of dominated convergence that 
\begin{align*} 
\lim \limits_{m \rightarrow \infty}|| (T(t_{m})\tau_{k}(u)) (\omega) - \overline{(\tau_{k}(u))}(\omega)||^{q}_{L^{q}(S)}=0 \text{ for $\P$-a.e. } \omega  \in \Omega.
\end{align*}
Hence, the boundedness of $T(t_{m})\tau_{k}(u)$, resp. $\overline{(\tau_{k}(u))}$, yield $\lim \limits_{m \rightarrow \infty}|| T(t_{m})\tau_{k}(u) - \overline{(\tau_{k}(u))}||_{L^{q}(\Omega;L^{q}(S))}=0$ and consequently
\begin{align}
\label{convproof2}
\lim \limits_{t \rightarrow \infty}|| T(t)\tau_{k}(u) - \overline{(\tau_{k}(u))}||_{L^{q}(\Omega;L^{q}(S))}=0,~\forall k \in (0,\infty).
\end{align}
Moreover, as $\lim \limits_{k \rightarrow \infty} \tau_{k}(u)=u$ in $L^{1}(\Omega;L^{1}(S))$ and as $|\tau_{k}(u)|\leq |u| \in L^{q}(\Omega;L^{q}(S))$ one verifies analogously that 
\begin{align}
\label{convproof3}
\lim \limits_{k \rightarrow \infty} ||\tau_{k}(u)-u||_{L^{q}(\Omega;L^{q}(S))}=0.
\end{align}
Now let $\varepsilon>0$ be arbitrary but fixed and observe that (\ref{convproof3}) yields the existence of a $k_{0} \in (0,\infty)$, such that 
\begin{align}
\label{convproof4}
\max (||\tau_{k_{0}}(u)-u||_{L^{q}(\Omega;L^{q}(S))},||\overline{(\tau_{k_{0}}(u))}-\overline{(u)}||_{L^{q}(\Omega;L^{q}(S))})< \frac{\varepsilon}{3},
\end{align}
and consequently, one has by invoking Proposition \ref{strongsolclosprop}.ii) that
\begin{align}
\label{convproof8}
||T(t)u-T(t)\tau_{k_{0}}(u)||_{L^{q}(\Omega;L^{q}(S))} < \frac{\varepsilon}{3},~\forall t \geq 0.
\end{align}
Moreover, (\ref{convproof2}) implies that there is $t_{0} \in [0,\infty)$, such that
\begin{align}
\label{convproof5}
|| T(t)\tau_{k_{0}}(u) - \overline{(\tau_{k_{0}}(u))}||_{L^{q}(\Omega;L^{q}(S))} < \frac{\varepsilon}{3},~\forall t \geq t_{0}.
\end{align}
Finally, one deduces the claim for $q \in (1,\infty)$ from (\ref{convproof4}), (\ref{convproof8}) and (\ref{convproof5}) analogously to the case $q=1$.
\end{proof}

\begin{remark} There are two kinds of more sophisticated decay estimates in \cite{ich}. Both of them depend on the relation between $p$ and $n$.\\
Firstly, assume $p>n$, let $v \in L^{p}(S)$ and $\gamma \in L^{1}_{g_{1},g_{2}}(S)$. Then one has $T(t,\gamma)v \in L^{\infty}(S)$ for every $t \in (0,\infty)$. Moreover, if in addition $\delta \in (n-1,p-1)$ is arbitrary but fixed, there is a constant $C^{\ast}_{S,\delta}$ such that
\begin{align*} 
||T(t,\gamma)v-\overline{(v)}||_{L^{\infty}(S)} \leq C^{\ast}_{ S,\delta} \lambda(S)^{\frac{1}{1+\delta}} \left( \frac{2}{\lambda(S)g_{1}|p-2|} \right)^{\frac{1}{p}} ||v-\overline{(v)}||^{\frac{2}{p}}_{L^{2}( S )} \left(\frac{1}{t}\right)^{\frac{1}{p}},~\forall t \in (0,\infty).
\end{align*}
In addition,  $C^{\ast}_{ S,\delta}$ can be chosen as $C^{\ast}_{ S ,\delta}=\tilde{C}_{ S ,1+\delta}\left( C_{ S ,1+\delta}^{1+\delta}+1\right)^{\frac{1}{1+\delta}}$, where $\tilde{C}_{ S,1+\delta}$ is the operator norm of the continuous injection $W^{1,1+\delta}( S )\hookrightarrow L^{\infty}( S )$, cf. \cite[Theorem 4.9]{ich}.\\
Consequently, one has, by invoking Theorem \ref{deterministicidentityprop}, the following: Assume $p>n$ and let $u \in L^{1}(\Omega;L^{1}(S))$ be such that $\P(u \in L^{p}(S))=1$. Then one has for any $\delta \in (n-1,p-1)$ and $t \in (0,\infty)$ that
\begin{align} 
\label{eqasunifbound}
||(T(t)u)(\omega)-\overline{(u)}(\omega)||_{L^{\infty}(S)} \leq C^{\ast}_{ S _{\gamma},\delta} \lambda(S)^{\frac{1}{1+\delta}} \left( \frac{2}{\lambda(S)g_{1}|p-2|} \right)^{\frac{1}{p}} \Delta_{u}(\omega)^{\frac{1}{p}} \left(\frac{1}{t}\right)^{\frac{1}{p}} 
\end{align}
for $\P$-a.e. $\omega \in  \Omega$, where $C^{\ast}_{S,\delta}$ can be chosen as above.\\
Inequality (\ref{eqasunifbound}) is a very strong result which gives almost surely an upper bound on the uniform distance between the solution and its limit. Unfortunately, this strong result is only valid if $p>n$.
\end{remark}

In the deterministic setting, one verifies that if $p$ is sufficiently small, then the solution extincts after finite time, cf. \cite[Theorem 5.7]{ich}. Unfortunately, this result does not carry on to the random case in a useful way: In general, one obtains a random time of extinction.\\
Nevertheless, one can prove strong results for small $p$. Doing so requires "slightly" more effort than for large $p$ and is treated in its own section.

\section{Decay estimates for $p \in [\frac{2n}{n+2},2)\setminus \{1\}$.}
\label{section_desmallp}
The purposes of this Section is to prove the estimates (\ref{introfuckingnicebound1}) and (\ref{introfuckingnicebound2}).

\begin{remark} Let $v \in L^{2}(S)$ and $\gamma \in L^{1}_{g_{1},g_{2}}(S)$. Throughout the remaining part of this section \linebreak$f_{v,\gamma}:[0,\infty) \rightarrow [0,\infty)$ denotes the function defined by
\begin{align*}
f_{v,\gamma}(t) := \log \left( \int \limits_{S} \left( T(t,\gamma)v-\overline{(v)} \right)^{2}d\lambda+1\right)
\end{align*}
for any $t \in [0,\infty)$.
\end{remark}

The basic technique to obtain a bound on the tail function of $||T(t)u-\overline{(u)}||^{2}_{L^{2}(S)}$ is as follows: One uses Markov's inequality to bound the tail function by $\frac{\mathbb{E}( \log(||T(t)u-u\overline{(u)}||^{2}_{L^{2}(S)}+1))}{\log(\alpha+1)}$. And afterwards one uses Theorem \ref{deterministicidentityprop} together with an upper bound on $f_{v,\gamma}$ to get an upper bound on the tail function of $||T(t)u-u\overline{(u)}||^{2}_{L^{2}(S)}$. Finally, some technical calculations yield the results (\ref{introfuckingnicebound1}) and (\ref{introfuckingnicebound2}).\\
The following well known lemma (which is a version of Gr\" onwall's inequality) builds the foundation for bounding $f_{v,\gamma}$.

\begin{lemma}\label{asympgenlemma} Let $f:[0,\infty)\rightarrow [0,\infty)$ be locally Lipschitz continuous. Moreover, set $b:=f(0)$ and assume that there is a $\beta > 0$ such that
\begin{align}
\label{niceextinctlemma}
f^{\prime}(t)+\beta f (t) \leq 0, \text{ for a.e. } t \in (0,\infty).
\end{align}
Then one has
\begin{align*}
f(t) \leq b \exp(-\beta t) 
\end{align*}
for all $t \in [0,\infty)$.
\end{lemma}

\begin{remark} Recall that $C_{S,q}$  denotes the Poincar\'{e} constant of $S$ in $L^{q}(S)$, $q \in [1,\infty)$. In addition, let $\tilde{C}_{S,\frac{2n}{n+2}}$ denote the operator norm of the continuous injection $W^{1,\frac{2n}{n+2}}(S) \hookrightarrow L^{2}(S)$. Note that $\frac{2n}{n+2} <n$, consequently the Sobolev embedding theorem yields the existence of such an injection. 
\end{remark}

\begin{lemma}\label{lcanddifflemma} Let $\gamma \in L^{1}_{g_{1},g_{2}}(S)$ and introduce $v \in D(a(\gamma))$. Then $f_{v,\gamma}$ is locally Lipschitz continuous. Moreover, one has $ T(t,\gamma)v \in W^{1,p}(S)$ for every $t \in (0,\infty)$ and
\begin{align}
\label{lcanddifflemmaeq}
f^{\prime}_{v,\gamma}(t)\leq -2 \int \limits_{S} \gamma |\nabla T(t,\gamma)v|^{p}d\lambda \left( \int \limits_{S}\left(v-\overline{(v)}\right)^{2}d\lambda +1 \right)^{-1}
\end{align}
for a.e. $t \in (0,\infty)$.
\end{lemma}
\begin{proof} At first the local Lipschitz continuity will be established. Let $\tau > 0$ be given. Then the mapping defined by $[0,\tau] \ni t \mapsto \int \limits_{S} \left( T(t,\gamma)v-\overline{(v)} \right)^{2}d\lambda$ is Lipschitz continuous, cf. \cite[Lemma 5.2]{ich}. This, together with the commonly known inequality
\begin{align*}
|\log(x+1)-\log(y+1)| \leq |x-y|,~ \forall x,y \in [0,\infty).
\end{align*}
yields the Lipschitz continuity of $f_{v,\gamma}|_{[0,\tau]}$.\\

As $v \in D(a(\gamma)) \subseteq L^{\infty}(S)$, one has by virtue of \cite[Lemma 3.3]{ich} that $T(t,\gamma)v \in D(a(\gamma))$ and a fortiori $ T(t,\gamma)v \in W^{1,p}(S)$, for all $t>0$. Consequently, it remains to prove (\ref{lcanddifflemmaeq}).\\
Firstly, one has, by  \cite[Lemma 5.3]{ich} that
\begin{align}
\label{lcproof2}
\frac{\partial}{\partial t} \int \limits_{S} \left( T(t,\gamma)v-\overline{(v)} \right)^{2}d\lambda = -2 \int \limits_{S} \gamma | \nabla T(t,\gamma)v|^{p}d\lambda,~\text{for a.e. } t \in (0,\infty).
\end{align}
Moreover, it follows from \cite[Lemma 4.1]{ich} and Remark \ref{remarkccsg} that
\begin{align*}
T(t,\gamma)v-\overline{(v)} = T(t,\gamma)\left(v-\overline{(v)}\right) << v-\overline{(v)},~\forall t \in [0,\infty)
\end{align*}
which implies
\begin{align*}
\int \limits_{S}\left(T(t,\gamma)v-\overline{(v)} \right)^{2} d \lambda \leq \int \limits_{S} \left(v-\overline{(v)}\right)^{2}d \lambda,~\forall t \in [0,\infty).
\end{align*}
This, together with (\ref{lcproof2}), yields
\begin{align*}
\frac{\partial}{\partial t} f_{v,\gamma}(t) =\frac{-2 \int \limits_{S} \gamma |\nabla T(t,\gamma)v|^{p}d\lambda}{ \int \limits_{S} \left(T(t,\gamma)v-\overline{(v)} \right)^{2}d\lambda+1} \leq -2 \int \limits_{S} \gamma |\nabla T(t,\gamma)v|^{p}d\lambda \left( \int \limits_{S} \left(v-\overline{(v)}\right)^{2}d \lambda+1\right)^{-1}
\end{align*}
for a.e. $t \in (0,\infty)$.
\end{proof}

\begin{lemma}\label{fancyestlemma} Let $\gamma \in L^{1}_{g_{1},g_{2}}(S)$, introduce $v \in D(a(\gamma))$ and set $m:= \frac{2n}{n+2}$. Moreover, assume \linebreak$p \in [m,2) \setminus \{1\}$. Then one has
\begin{align*}
f_{v,\gamma}(t) \leq -f^{\prime}_{v,\gamma}(t) \frac{1}{p} \max \left( \tilde{C}^{2}_{S,m}\left( C_{S,m}^{m}+1\right)^{\frac{2}{m}}g_{1}^{-\frac{2}{p}}\lambda(S)^{\frac{p-m}{p}},1\right) \left( \int \limits_{S} \left(v-\overline{(v)}\right)^{2}d \lambda+1\right),
\end{align*} 
for a.e. $t \in (0,\infty)$.
\end{lemma}
\begin{proof} Firstly, one infers by combining the Sobolev embedding theorem and $\overline{(T(t,\gamma)v)}=\overline{(v)}$ that 
\begin{align*}
\int \limits_{S} \left( T(t,\gamma)v-\overline{(v)}\right)^{2}d \lambda \leq \tilde{C}^{2}_{S,m} \left( ||T(t,\gamma)v-\overline{(v)}||^{m}_{L^{m}(S)}+||\nabla T(t,\gamma)v||^{m}_{L^{m}(S;\mathbb{R}^{n})}\right)^{\frac{2}{m}}.
\end{align*}
Using this and Poincar\'{e}'s inequality yields
\begin{align}
\label{decayhelplemmaproof1}
\int \limits_{S} \left( T(t,\gamma)v-\overline{(v)}\right)^{2}d \lambda \leq \tilde{C}^{2}_{S,m}\left( C_{S,m}^{m}+1\right)^{\frac{2}{m}} \left( \int \limits_{S} |\nabla T(t,\gamma)v|^{m}d\lambda \right)^{\frac{2}{m}},~\forall t \in [0,\infty),
\end{align}
which is finite as $p \geq m$ and $T(t,\gamma)v \in W^{1,p}(S)$, by Lemma \ref{lcanddifflemma}.\\
Consequently, one has by observing that $p \geq m$ and $\gamma \geq g_1$ and by applying (\ref{decayhelplemmaproof1}) as well as H\"older's inequality that
\begin{align*}
f_{v,\gamma}(t) \leq \log \left( \tilde{C}^{2}_{S,m}\left( C_{S,m}^{m}+1\right)^{\frac{2}{m}} g_{1}^{-\frac{2}{p}}\lambda(S)^{\frac{p-m}{p}} \left( \int \limits_{S} {\gamma}| \nabla T(t,\gamma)v|^{p}d\lambda \right)^{\frac{2}{p}}+1\right),~\forall t \in [0,\infty).
\end{align*} 
Now it is plain that  $\tilde{C}^{2}_{S,m}\left( C_{S,m}^{m}+1\right)^{\frac{2}{m}} g_{1}^{-\frac{2}{p}}\lambda(S)^{\frac{p-m}{p}} \leq \max( \tilde{C}^{2}_{S,m}\left( C_{S,m}^{m}+1\right)^{\frac{2}{m}} g_{1}^{-\frac{2}{p}}\lambda(S)^{\frac{p-m}{p}},1) $ and hence employing Bernoulli's inequality yields
\begin{align}
\label{decayhelplemmaproof2}
f_{v,\gamma}(t) \leq \max \left( \tilde{C}^{2}_{S,m}\left( C_{S,m}^{m}+1\right)^{\frac{2}{m}}g_{1}^{-\frac{2}{p}}\lambda(S)^{\frac{p-m}{p}},1\right) \log \left( \left(\int \limits_{S} \gamma | \nabla T(t,\gamma)v|^{p}d\lambda \right)^{\frac{2}{p}}+1\right)
\end{align}
for all $t \in [0,\infty)$.\\
Consequently, one has by (\ref{decayhelplemmaproof2}) and by using the well known inequalities $x^{\frac{2}{p}}+1=x^{\frac{2}{p}}+1^{\frac{2}{p}} \leq (x+1)^{\frac{2}{p}}$ and $\log(x+1) \leq x$ for all $x\geq 0$ that  
\begin{align*}
f_{v,\gamma}(t) \leq  \max \left( \tilde{C}^{2}_{S,m}\left( C_{S,m}^{m}+1\right)^{\frac{2}{m}}g_{1}^{-\frac{2}{p}}\lambda(S)^{\frac{p-m}{p}},1\right)\frac{2}{p}\int \limits_{S} \gamma |\nabla T(t,\gamma)v|^{p}d\lambda.
\end{align*}
Finally, one infers the claim from the last inequality and (\ref{lcanddifflemmaeq}).
\end{proof}

\begin{lemma}\label{determinisiticfinallemma} Let $\gamma \in L^{1}_{g_{1},g_{2}}(S)$, introduce $v \in L^{2}(S)$, set $m:= \frac{2n}{n+2}$ and assume $p \in [m,2) \setminus \{1\}$. Then one has
\begin{align}
\label{determinisiticfinaleq}
\log \left( \int \limits_{S} \left( T(t,\gamma)v-\overline{(v)}\right)^{2}d\lambda+1\right) \leq 2||v-\overline{(v)}||_{L^{2}(S)} \exp \left(\frac{-C^{\ast}_{S,m,p,g_{1}}t}{1+||v-\overline{(v)}||^{2}_{L^{2}(S)}}\right)
\end{align}
for every $t \in [0,\infty)$, where
\begin{align}
\label{determinisiticfinaleqc}
C^{\ast}_{S,m,p,g_{1}} =  p \left(\max \left( \tilde{C}^{2}_{S,m}\left( C_{S,m}^{m}+1\right)^{\frac{2}{m}}g_{1}^{-\frac{2}{p}}\lambda(S)^{\frac{p-m}{p}},1\right)\right)^{-1}.
\end{align}
\end{lemma}
\begin{proof} Firstly, assume $v \in D(a(\gamma))$ and introduce $\beta:= C^{\ast}_{S,m,p,g_{1}} \left( 1+||v-\overline{(v)}||^{2}_{L^{2}(S)} \right)^{-1}$. Then one has, by recalling Lemma \ref{fancyestlemma}, that $f^{\prime}_{v,\gamma}(t)+\beta f_{v,\gamma}(t) \leq 0$ for a.e. $t \in (0,\infty)$ which yields, by invoking Lemma \ref{asympgenlemma} that $f_{v,\gamma}(t) \leq f_{v,\gamma}(0)\exp(-\beta t)$ for every $t \in [0,\infty)$ and therefore 
\begin{align}
\label{determinisiticfinalproof1}
\log \left( \int \limits_{S} \left( T(t,\gamma)v-\overline{(v)}\right)^{2}d\lambda+1\right) \leq \log \left( \int \limits_{S} \left( v-\overline{(v)}\right)^{2}d\lambda+1\right)\exp \left(\frac{-C^{\ast}_{S,m,p,g_{1}}t}{1+||v-\overline{(v)}||^{2}_{L^{2}(S)}}\right).
\end{align}
Consequently, as $\log(x^{2}+1) \leq 2x$ for all $x \geq 0$ one obtains
\begin{align}
\label{determinisiticfinalproof2}
\log \left( \int \limits_{S} \left( v-\overline{(v)}\right)^{2}d\lambda+1\right)  = \log(||v-\overline{(v)}||_{L^{2}(S)}^{2}+1) \leq 2 ||v-\overline{(v)}||_{L^{2}(S)}.
\end{align}
Hence, one has, by combining (\ref{determinisiticfinalproof1}) and (\ref{determinisiticfinalproof2}) that (\ref{determinisiticfinaleq}) holds, if $v \in D(a(\gamma))$.\\
Now let $v \in L^{2}(S)$ and introduce $(v_{k})_{k \in \mathbb{N}} \subseteq D(a(\gamma))$ such that $\lim \limits_{k \rightarrow \infty}v_{k}=v$ in $L^{2}(S)$. Such a sequence exists, cf. \cite[Lemma 5.6]{ich}.\\
Then trivially $\lim \limits_{k \rightarrow \infty}\overline{(v_{k})} = \overline{(v)}$ and moreover, one has by contractivity (cf. Remark \ref{remarksemicontracexpf}.ii)) that $\lim \limits_{k \rightarrow \infty} T(t,\gamma)v_{k} = T(t,\gamma)v$ in $L^{2}(S)$.\\
As the mappings $[0,\infty) \ni x \mapsto \log(x+1)$ and $[0,\infty) \ni x \mapsto \exp(-(x+1)^{-1}C^{\ast}_{S,m,p,g_{1}}t)$ are continuous, the claim follows.
\end{proof}

\begin{remark} In the sequel $C^{\ast}_{S,m,p,g_{1}}$ denotes the constant defined in (\ref{determinisiticfinaleqc}). The previous lemma brings us in the position to prove the main result of this section.
\end{remark}

\begin{theorem}\label{fuckingbeautifultheorem} Let $u \in L^{1}(\Omega;L^{1}(S))$, $t \in (0,\infty)$, $\alpha >0$ and assume that $p \in [m,2) \setminus \{1\}$, where $m:=\frac{2n}{n+2}$. Then all of the following assertions hold.\\ 
If $\Delta_{u} \in L^{1}(\Omega)$, one has
\begin{align}
\label{fuckingbeautifultheoremeq1}
\P \left( \int \limits_{S} (T(t)u-\overline{(u)})^{2}d\lambda > \alpha \right) \leq \frac{2}{\log(\alpha+1)}~ \left(\E(\Delta_{u}) \E \left(\exp \left(\frac{-2tC^{\ast}_{S,m,p,g_{1}}}{1+\Delta_{u}}\right)\right) \right)^{\frac{1}{2}}.
\end{align}
If $r \in [1,\infty)$ and $\Delta_{u} \in L^{2r}(\Omega)$, one has
\begin{align}
\label{fuckingbeautifultheoremeq2}
\P \left( \int \limits_{S} (T(t)u-\overline{(u)})^{2}d\lambda > \alpha \right) \leq \left(\frac{1}{t}\right)^{r} \frac{2}{\log(\alpha+1)} \left( \frac{r}{2C^{\ast}_{S,m,p,g_{1}}}\right)^{r} \left(\E(\Delta_{u})\E((1+\Delta_{u})^{2r})\right)^{\frac{1}{2}}.
\end{align}
If there is an $\varepsilon>0$ such that $e^{\varepsilon \Delta_{u}} \in L^{1}(\Omega)$, one has
\begin{align}
\label{fuckingbeautifultheoremeq3}
\P \left( \int \limits_{S} (T(t)u-\overline{(u)})^{2}d\lambda > \alpha \right) \leq \exp\left(-t^{\frac{1}{2}}\left(\frac{\varepsilon C^{\ast}_{S,m,p,g_{1}}}{2}\right)^{\frac{1}{2}}\right)\frac{2\exp(\frac{\varepsilon}{2})}{\log(\alpha+1)} \left(\E(\Delta_{u})\E \left(\exp \left(\varepsilon\Delta_{u}\right)\right)\right)^{\frac{1}{2}} .
\end{align} 
\end{theorem}
\begin{proof} Proof of (\ref{fuckingbeautifultheoremeq1}). \\
Firstly, note that $[0,\infty) \ni x \mapsto \log(x+1)$ is  obviously nonnegative, increasing and strictly positive on $(0,\infty)$. Consequently, one has by virtue of Markov's inequality and by recalling Lemma \ref{determinisiticfinallemma} as well as Theorem \ref{deterministicidentityprop} that 
\begin{align*}
\P \left( \int \limits_{S} (T(t)u-\overline{(u)})^{2}d\lambda > \alpha \right) \leq\frac{2}{\log(\alpha+1)} \int \limits_{\Omega} \Delta_{u}(\omega)^{\frac{1}{2}} \exp \left(\frac{-C^{\ast}_{S,m,p,g_{1}}t}{1+\Delta_{u}(\omega)}\right) d \P(\omega)
\end{align*}
which verifies (\ref{fuckingbeautifultheoremeq1}) by applying Cauchy-Schwarz' inequality. (Moreover, note that the assumption on $\Delta_u$ ensures that the first expectation exists and the the second one exists trivially.)\\

\noindent Throughout the remaining part of this proof, let $\tilde{\Delta}_{u}:=\frac{1}{1+\Delta_{u}}$.\\
Now inequality (\ref{fuckingbeautifultheoremeq2}) follows from the succeeding estimate, where relation (\ref{fuckingbeautifultheoremeq1}) is used.
\begin{eqnarray*}
	& & ~
t^{r}\P \left( \int \limits_{S} (T(t)u-\overline{(u)})^{2}d\lambda > \alpha \right) \\
& \leq & ~ t^{r}  \frac{2}{\log(\alpha+1)}~ \left(\E(\Delta_{u}) \E \left(\exp \left(\frac{-2tC^{\ast}_{S,m,p,g_{1}}}{1+\Delta_{u}}\right)\right) \right)^{\frac{1}{2}}\\ 
& = & ~   \frac{2(\E(\Delta_{u}))^{\frac{1}{2}}}{\log(\alpha+1)}~ \left(\E \left( (2r^{-1}\tilde{\Delta}_{u}C^{\ast}_{S,m,p,g_{1}})^{-2r}\exp \left( r \log \left((2r^{-1}\tilde{\Delta}_{u}C^{\ast}_{S,m,p,g_{1}}t)^{2}\right)-2t C^{\ast}_{S,m,p,g_{1}}\tilde{\Delta}_{u}\right)\right) \right)^{\frac{1}{2}} \\
& \leq & ~   \frac{2(\E(\Delta_{u}))^{\frac{1}{2}}}{\log(\alpha+1)}~ \left( \E \left( (2r^{-1}\tilde{\Delta}_{u}C^{\ast}_{S,m,p,g_{1}})^{-2r}\exp \left( r  2r^{-1}\tilde{\Delta}_{u}C^{\ast}_{S,m,p,g_{1}}t-2t C^{\ast}_{S,m,p,g_{1}}\tilde{\Delta}_{u}\right)\right)\right) ^{\frac{1}{2}} \\
& = & ~   \frac{2}{\log(\alpha+1)} \left(\frac{r}{2C^{\ast}_{S,m,p,g_{1}}}\right)^{r}~ \left(\E(\Delta_{u}) \E \left( (1+\Delta_{u})^{2r}\right) \right)^{\frac{1}{2}}.
\end{eqnarray*}
Finally, (\ref{fuckingbeautifultheoremeq3}) will be proven. For the sake of brevity, let $\beta:=\left(\frac{1}{2}\varepsilon C^{\ast}_{S,m,p,g_{1}}\right)^{\frac{1}{2}}$. The estimate follows from the succeeding calculation, where relation (\ref{fuckingbeautifultheoremeq1}) is used.
\begin{eqnarray*}
	& & ~
\exp\left(t^{\frac{1}{2}}\beta\right)\P \left( \int \limits_{S} (T(t)u-\overline{(u)})^{2}d\lambda > \alpha \right) \\
& \leq & \frac{2\left(\E(\Delta_{u})\right)^{\frac{1}{2}}}{\log(\alpha+1)} \left(\E \left(\exp \left(2t^{\frac{1}{2}}\left(\beta- t^{\frac{1}{2}} C^{\ast}_{S,m,p,g_{1}}\tilde{\Delta}_{u}\right)\id\{\beta>t^{\frac{1}{2}} C^{\ast}_{S,m,p,g_{1}}\tilde{\Delta}_{u}\}\right)\right) \right)^{\frac{1}{2}} \\
& \leq & \frac{2\left(\E(\Delta_{u})\right)^{\frac{1}{2}}}{\log(\alpha+1)} \left(\E \left(\exp \left(2t^{\frac{1}{2}}\beta\id\left\{t^{\frac{1}{2}}~< \frac{\beta}{C^{\ast}_{S,m,p,g_{1}}\tilde{\Delta}_{u}}\right\}\right)\right) \right)^{\frac{1}{2}} \\ 
& \leq & \frac{2\left(\E(\Delta_{u})\right)^{\frac{1}{2}}}{\log(\alpha+1)} \left(\E \left(\exp \left(2~\frac{\beta^{2}}{C^{\ast}_{S,m,p,g_{1}}}(1+\Delta_{u})\right)\right) \right)^{\frac{1}{2}} \\ 
& = & ~\frac{2}{\log(\alpha+1)}\exp\left(\frac{\varepsilon}{2}\right)\left(\E(\Delta_{u})\E \left(\exp \left(\varepsilon\Delta_{u}\right)\right) \right)^{\frac{1}{2}}
\end{eqnarray*}
\end{proof}

\begin{remark} If  $u:\Omega \rightarrow L^{1}(S)$ is Gaussian and $\P(u \in L^{2}(S))=1$, it is clear that there is an $\varepsilon >0$, such that $e^{\varepsilon\Delta_{u}} \in L^{1}(\Omega,\F,\P)$. Consequently, one can apply (\ref{fuckingbeautifultheoremeq3}) if $u$ is Gaussian and  $\P(u \in L^{2}(S))=1$. 
\end{remark}

\begin{remark} Note that if $n=2$, one can apply Theorem \ref{fuckingbeautifultheorem} for any $p \in (1,2)$. Moreover, one can apply (\ref{eqasunifbound}), if $p>2$. Consequently, if $n=2$, which is the interesting case from an applied point of view, one can apply one of these two results, given that the initial $u$ is sufficiently integrable.
\end{remark}

\renewcommand{\chaptername}{}
\renewcommand{\thechapter}{}
\chapter{Appendices}
\renewcommand{\thesection}{\Alph{section}}
\numberwithin{equation}{section}
\thispagestyle{empty}
\markboth{Kapitelname}{APPENDIX}

\section{Measurability questions concerning $A$}
\label{appendix_A}
The following two lemmas reveal that all events occurring in the definition of $A$ are indeed measurable and that all occurring integrals are well-defined as well as finite.

\begin{lemma}\label{lemmawelldefined1} The set $W^{1,p}(S)\cap L^{\infty}(S)$ is $\mathfrak{B}(L^{1}(S))$-measurable. Let $f \in L^{1}(\Omega;L^{1}(S))$ and assume $P(f \in W^{1,p}(S)\cap L^{\infty}(S))=1$. Then the following assertions hold.
	\begin{enumerate}
		\item $f$ is $\mathcal{F}$-$\mathfrak{B}(L^{p}(S))$-measurable.
		\item $\nabla f$ is $\mathcal{F}$-$\mathfrak{B}(L^{p}(S;\mathbb{R}^{n}))$-measurable.
		\item The mapping $\Phi: L^{p}(S;\mathbb{R}^{n})\rightarrow L^{\tp}(S;\mathbb{R}^{n})$ defined by $\Phi(\varphi):=|\varphi|^{p-2}\varphi$ for all $\varphi \in L^{p}(S;\mathbb{R}^{n})$ is continuous.
		\item The mapping defined by $\Omega \ni \omega \mapsto g(\omega)|\nabla f(\omega)|^{p-2}\nabla f(\omega) \in L^{\tp}(S;\mathbb{R}^{n})$ is $\mathcal{F}$-$\mathfrak{B}(L^{\tp}(S;\mathbb{R}^{n}))$-measurable.
	\end{enumerate}
\end{lemma}
\begin{proof} At first it will be proven that  $L^{\infty}(S) \in \mathfrak{B}(L^{1}(S))$.\\
	Introduce $K(\kappa):=\{f \in L^{\infty}(S):~||f||_{L^{\infty}(S)}\leq \kappa \}$ and note that obviously $L^{\infty}(S)=  \bigcup \limits_{\kappa \in \mathbb{N}} K(\kappa)$ and that each of the $K(\kappa)$ is closed w.r.t. $||\cdot||_{L^{1}(S)}$. Consequently, $L^{\infty}(S)$ is the countable union of $L^{1}(S)$-closed sets and therefore $ L^{\infty}(S) \in \mathfrak{B}(L^{1}(S))$.\\
	Now it will be proven that $W^{1,p}(S) \in \mathfrak{B}(L^{1}(S))$. Introduce $K_{p}(\kappa):= \{ f \in W^{1,p}(S):~||f||_{W^{1,p}(S)} \leq \kappa \}$, then clearly $W^{1,p}(S)=\bigcup \limits_{\kappa \in \mathbb{N}}K_{p}(\kappa)$. Consequently, the claim follows if $K_{p}(\kappa)$ is $L^{1}(S)$-closed.\\
	Let $(f_{m})_{m \in \mathbb{N}} \subseteq K_{p}(\kappa)$ and $f \in L^{1}(S)$ be such that $\lim \limits_{m \rightarrow \infty}f_{m}=f$ in $L^{1}(S)$. Firstly, $||\nabla f_{m}||_{L^{p}(S;\mathbb{R}^{n})} \leq \kappa$ for all $m \in \mathbb{N}$, which yields, by passing to a subsequence if necessary, that there is an \linebreak$F:=(F_{1},...,F_{n}) \in L^{p}(S;\mathbb{R}^{n})$, with
	\begin{align*} 
	\wlim \limits_{m \rightarrow \infty} \nabla f_{m}=F \text{, in } L^{p}(S;\mathbb{R}^{n}).
	\end{align*}
	The latter, together with $\lim \limits_{m \rightarrow \infty}f_{m}=f$ in $L^{1}(S)$, yields that one has for all $\varphi \in C_{c}^{\infty}(S)$ that
	\begin{align*}
	\int \limits_{S} f \frac{\partial}{\partial x_{j}} \varphi d \lambda = \lim \limits_{m \rightarrow \infty} \int \limits_{S} f_{m} \frac{\partial}{\partial x_{j}} \varphi d \lambda = -\lim \limits_{m \rightarrow \infty} \int \limits_{S} \varphi \frac{\partial}{\partial x_{j}} f_{m} d \lambda = -\int \limits_{S} \varphi F_{j} d \lambda.
	\end{align*}
	Hence $f \in W^{1,1}_{\text{loc}}(S)$ and $\nabla f=F$. Consequently, one also has $\nabla f \in L^{p}(S;\mathbb{R}^{n})$. Moreover, one has (by passing to a subsequence if necessary) that $\lim \limits_{m \rightarrow \infty} f_{m}=f$ a.e. on $S$. It follows by virtue of Fatou's Lemma that
	\begin{align*} 
	\int \limits_{S} |f|^{p} d \lambda \leq \liminf_{m \rightarrow \infty} \int \limits_{S} |f_{m}|^{p}d \lambda \leq \kappa^{p}.
	\end{align*}
	Consequently, $f \in W^{1,p}(S)$.\\
	Finally, one has
	\begin{align*}
		||f||_{W^{1,p}(S)}^{p}\leq \liminf \limits_{m \rightarrow \infty} \left( \int \limits_{S} |f_{m}|^{p}d \lambda + \int \limits_{S} | \nabla f_{m}| ^{p} d \lambda \right) = \liminf \limits_{m \rightarrow \infty}||f_{m}||_{W^{1,p}(S)}^{p}= \kappa^{p}
	\end{align*}
	and consequently $f \in K_{p}(\kappa)$.\\
	As $L^{\infty}(S),~W^{1,p}(S) \in \mathfrak{B}(L^{1}(S))$, it is clear that $L^{\infty}(S) \cap W^{1,p}(S) \in \mathfrak{B}(L^{1}(S))$.\\  
	Proof of i). By assumption $f$ is $\mathcal{F}$-$\mathfrak{B}(L^{1}(S))$-measurable. Hence the mapping
	\begin{align*}
	\Omega \ni \omega \mapsto \int \limits_{S} f(\omega)\varphi d \lambda
	\end{align*}
	is $\mathcal{F}$-$\mathfrak{B}(\mathbb{R})$-measurable for any $\varphi \in L^{\infty}(S)$.\\
	Now let $\varphi \in L^{\tp}(S)$ and $(\varphi_{m})_{m \in \mathbb{N}} \subseteq L^{\infty}(S)$ such that $\lim \limits_{m \rightarrow \infty} \varphi_{m}=\varphi$ in $L^{\tp}(S)$. Then
	\begin{align*}
	\Omega \ni \omega \mapsto \int \limits_{S} f(\omega)\varphi d \lambda = \lim \limits_{m \rightarrow \infty} \int \limits_{S
	} f(\omega) \varphi_{m} d \lambda,
	\end{align*}
	for $\P$-a.e. $\omega \in \Omega$, since $\P(f \in L^{p}(S))=1$. Hence $\Omega \ni \omega \mapsto \int \limits_{S} f(\omega)\varphi d \lambda$ is, as it is the almost sure limit of $\F$-$\mathfrak{B}(\mathbb{R})$-measurable functions, $\mathcal{F}$-$\mathfrak{B}(\mathbb{R})$-measurable. Consequently, Pettis' measurability theorem yields i).\\
	Proof of ii). Firstly, note that $\P(f \in W^{1,p}(S))=1$, consequently $\nabla f(\omega)$ exists for  $\P$-a.e. $\omega \in \Omega$.\\
	For $\varphi=(\varphi_{1},...,\varphi_{n}) \in C_{c}^{\infty}(S;\mathbb{R}^{n})$, one has
	\begin{align*}
	\int \limits_{S}  \nabla f(\omega)\cdot \varphi d\lambda = - \int \limits_{S} f(\omega) \sum \limits_{j=1} \limits^{n} \frac{\partial}{\partial x_{j}} \varphi_{j} d \lambda
	\end{align*}
	for $\P$-a.e. $\omega \in \Omega$.\\
	Consequently, $\Omega \ni \omega \mapsto \int \limits_{S}  \nabla f(\omega)\cdot \varphi d\lambda$ is, for any $\varphi \in C_{c}^{\infty}(S;\mathbb{R}^{n})$, $\mathcal{F}$-$\mathfrak{B}(\mathbb{R})$-measurable because $f$ is $\F$-$\B(L^{1}(S))$-measurable and $\sum \limits_{j=1} \limits^{n} \frac{\partial}{\partial x_{j}} \varphi_{j} \in L^{\infty}(S)$.\\
	Now let $\varphi \in L^{\tp}(S;\mathbb{R}^{n})$ and let $(\varphi_{m})_{m \in \mathbb{N}} \subseteq  C_{c}^{\infty}(S;\mathbb{R}^{n})$ be such that $\lim \limits_{m \rightarrow \infty} \varphi_{m}=\varphi$ in $L^{\tp}(S;\mathbb{R}^{n})$. Then, as $\nabla f \in L^{p}(S;\mathbb{R}^{n})$ a.s., one obtains that 
	\begin{align*}
	\int \limits_{S}  \nabla f(\omega)\cdot \varphi d\lambda = \lim \limits_{m \rightarrow \infty} \int \limits_{S}  \nabla f(\omega)\cdot \varphi_{m} d\lambda \text{ for $\P$-a.e. } \omega \in \Omega.
	\end{align*}
	Consequently, $\Omega \ni \omega \mapsto \int \limits_{S}  \nabla f(\omega)\cdot \varphi d\lambda$ is, for any $\varphi \in L^{\tp}(S;\mathbb{R}^{n})$, the almost sure limit of measurable functions and hence itself $\F$-$\B(\mathbb{R})$-measurable. Conclusively, it follows by virtue of Pettis' measurability theorem that $\nabla f$ is $\F$-$\B(L^{p}(S;\mathbb{R}^{n}))$-measurable.\\  
	Proof of iii). Let $\varphi \in L^{p}(S;\mathbb{R}^{n})$ and let $(\varphi_{m})_{m \in \mathbb{N}} \subseteq L^{p}(S;\mathbb{R}^{n})$, such that $\lim \limits_{m \rightarrow \infty} \varphi_{m}=\varphi$ in $L^{p}(S;\mathbb{R}^{n})$. \\
	One has, by passing to a subsequence if necessary, that there is an $h \in L^{p}(S)$ such that
	$\lim \limits_{m \rightarrow \infty} \varphi_{m}=\varphi$ a.e. on $S$ and $| \varphi_{m}| \leq |h|$ a.e. on $S$ for each $m \in \mathbb{N}$. Moreover, the continuity of $\mathbb{R}^{n} \ni x \mapsto | x|^{p-2}x$ yields
	\begin{align} 
	\lim \limits_{m \rightarrow \infty} \Phi(\varphi_{m})=\Phi(\varphi)~\text{a.e. on } S.
	\end{align}
	In addition, 
	\begin{align*}
	| \Phi(\varphi_{m})-\Phi(\varphi)|^{\tp} \leq (| \varphi_{m}|^{p-1}+| \varphi|^{p-1})^{\tp} \leq 2^{\tp} |h|^{p} \in L^{1}(S),~\forall m \in \mathbb{N}.
	\end{align*}
	This yields, by virtue of dominated convergence, that $\lim \limits_{m \rightarrow \infty} \Phi(\varphi_{m})=\Phi(\varphi)$ in $L^{\tp}(S;\mathbb{R}^{n})$. \\ 
	Proof of iv). It is obvious that $\P( g \in L^{p}(S))=1$. In addition, $g$ is by assumption $\F$-$\mathfrak{B}(L^{1}(S))$-measurable. Consequently, one has by working as in i) that $g$ is $\F$-$\mathfrak{B}(L^{p}(S))$-measurable. Moreover, ii) and iii) yield that $\Omega \ni \omega \mapsto | \nabla f(\omega)|^{p-2}\nabla f(\omega)$ is $\F$-$L^{\tp}(S;\mathbb{R}^{n})$-measurable.\\ 
	Consequently, one infers that $g|\nabla f|^{p-2}\nabla f$ is $\F$-$\mathfrak{B}(L^{1}(S;\mathbb{R}^{n}))$-measurable.\\
	In addition, as $|\nabla f|^{p-2}\nabla f \in L^{\tp}(S;\mathbb{R}^{n})$ a.s. and particularly $g \in L^{\infty}(S)$ almost surely, one has that $g|\nabla f|^{p-2}\nabla f \in L^{\tp}(S;\mathbb{R}^{n})$ a.s. This, together with the $\F$-$\mathfrak{B}(L^{1}(S;\mathbb{R}^{n}))$-measurability of $g|\nabla f|^{p-2}\nabla f$ yields, by working similarly to i), that $g|\nabla f|^{p-2}\nabla f$ is $\F$-$\mathfrak{B}(L^{\tp}(S;\mathbb{R}^{n}))$-measurable.
\end{proof}

\begin{lemma}\label{lemmawelldefined2}  Let $f,~\hat{f} \in L^{1}(\Omega;L^{1}(S))$ and assume $P(f \in W^{1,p}(S)\cap L^{\infty}(S))=1$. Then the Lebesgue integrals 
	\begin{align}
	\label{asslemma1}
	\int \limits_{S}g(\omega)|\nabla f(\omega)|^{p-2} \nabla f(\omega)\cdot \nabla \varphi d\lambda\text{ and } \int \limits_{S}\hat{f}(\omega)\varphi d \lambda
	\end{align}
	exist for any given $\varphi \in W^{1,p}(S)\cap L^{\infty}(S)$ and a.e. $\omega \in \Omega$. Moreover, the mappings defined by
	\begin{align}
	\label{asslemma2}
	\Omega \ni \omega \mapsto \int \limits_{S}g(\omega)|\nabla f(\omega)|^{p-2} \nabla f(\omega)\cdot \nabla \varphi d\lambda\text{ and } \Omega \ni \omega \mapsto \int \limits_{S}\hat{f}(\omega)\varphi d \lambda
	\end{align}
	are $\mathcal{F}$-$\mathfrak{B}(\mathbb{R})$-measurable for any $\varphi \in W^{1,p}(S)\cap L^{\infty}(S)$.\\
	Finally, one has
	\begin{align*}
	\left\{ \omega \in \Omega:~ \int \limits_{S}g(\omega)|\nabla f(\omega)|^{p-2} \nabla f(\omega)\cdot \nabla \varphi d\lambda=\int \limits_{S}\hat{f}(\omega)\varphi d \lambda,~\forall \varphi \in W^{1,p}(S)\cap L^{\infty}(S)\right\} \in \mathcal{F},
	\end{align*}
	and $A$ is well-defined. 
\end{lemma}
\begin{proof} Firstly, note that the assertions concerning $\hat{f}$ stated in (\ref{asslemma1}) and (\ref{asslemma2}) are trivial.\\ 
	Moreover, for $\varphi \in W^{1,p}(S)\cap L^{\infty}(S)$, one has a fortiori $\nabla \varphi \in L^{p}(S;\mathbb{R}^{n})$ which yields, by virtue of Lemma \ref{lemmawelldefined1}.iv), that the left-hand-side integral in (\ref{asslemma1}) exists with probability one and also that the left-hand-side mapping in (\ref{asslemma2}) is $\F$-$\B(\mathbb{R})$-measurable.\\
	Now the final assertion in this lemma will be proven. Firstly, note that
	\begin{align}
	\label{assumplemma7}
	\left\{ \omega \in \Omega:~ \int \limits_{S}g(\omega)|\nabla f(\omega)|^{p-2} \nabla f(\omega)\cdot \nabla \varphi d\lambda=\int \limits_{S}\hat{f}(\omega)\varphi d \lambda\right\} \in \mathcal{F}
	\end{align}
	for any given $\varphi \in W^{1,p}(S) \cap L^{\infty}(S)$.\\
	Introduce $L^{\infty}_{k}(S):=\{ f \in L^{\infty}(S):~||f||_{L^{\infty}(S)} \leq k \}$ for every $k \in \mathbb{N}$.
	One verifies immediately that $W^{1,p}(S) \cap L^{\infty}_{k}(S)$ is a closed subset of $W^{1,p}(S)$ w.r.t. $||\cdot||_{W^{1,p}(S)}$ . Moreover, it is well known that $(W^{1,p}(S),||\cdot||_{W^{1,p}(S)})$ is separable and that subsets of separable spaces are separable as well. Consequently, for each $k \in \mathbb{N}$ there is a countable set $\mathcal{D}(k) \subseteq W^{1,p}(S) \cap L^{\infty}_{k}(S)$ fulfilling
	\begin{align*}
	\overline{\mathcal{D}(k)}=  W^{1,p}(S) \cap L^{\infty}_{k}(S),
	\end{align*}
	where the closure is taken w.r.t. $||\cdot||_{W^{1,p}(S)}$.\\ 
	Now introduce $\gamma \in L^{1}_{g_{1},g_{2}}(S)$ and $(F,\hat{F}) \in (W^{1,p}(S)\cap L^{\infty}(S)) \times L^{1}(S)$. It will be proven that one has, for a given $k \in \mathbb{N}$,
	\begin{align}
	\label{assumplemma4}
	\int \limits_{S}\gamma| \nabla F|^{p-2}\nabla F \cdot \nabla \varphi d\lambda=\int \limits_{S}\hat{F}\varphi d \lambda,~\forall \varphi \in W^{1,p}(S) \cap L^{\infty}_{k}(S),
	\end{align}
	if and only if
	\begin{align}
	\label{assumplemma5}
	\int \limits_{S}\gamma| \nabla F|^{p-2} \nabla F\cdot \nabla \varphi d\lambda=\int \limits_{S}\hat{F}\varphi d \lambda,~\forall \varphi \in \mathcal{D}(k).
	\end{align}
	Firstly, note that (\ref{assumplemma4}) obviously implies (\ref{assumplemma5}).\\
	Now assume that (\ref{assumplemma5}) holds. Let $\varphi \in W^{1,p}(S) \cap L^{\infty}_{k}(S)$ be arbitrary but fixed and introduce\linebreak$(\varphi_{m})_{m \in \mathbb{N}} \subseteq \mathcal{D}(k)$ such that $\lim \limits_{m \rightarrow \infty} \varphi_{m}=\varphi$ in $W^{1,p}(S)$.\\
	As $\gamma|\nabla F|^{p-2}\nabla F \in L^{\tp}(S;\mathbb{R}^{n})$ and as particularly $\lim \limits_{m \rightarrow \infty} \nabla \varphi_{m}=\nabla \varphi$ in $L^{p}(S;\mathbb{R}^{n})$ one obtains that
	\begin{align*}
	\lim \limits_{m \rightarrow \infty} \int \limits_{S}\gamma|\nabla F|^{p-2} \nabla F\cdot \nabla \varphi_{m}d\lambda = \int \limits_{S} \gamma|\nabla F|^{p-2} \nabla F\cdot \nabla \varphi d\lambda.
	\end{align*}
	Moreover, as particularly  $\lim \limits_{m \rightarrow \infty} \varphi_{m}=\varphi$ in $L^{p}(S)$ one obtains, by passing to a subsequence if necessary, that $\lim \limits_{m \rightarrow \infty} (\varphi_{m}-\varphi)\hat{F}=0$ a.e. on $S$. Since clearly $|(\varphi_{m}-\varphi)\hat{F}| \leq 2k |\hat{F}|\in L^{1}(S)$ one obtains by virtue of dominated convergence that
	\begin{align*}
	\lim \limits_{m \rightarrow \infty} \int \limits_{S}\hat{F}\varphi_{m} d \lambda = \int \limits_{S}\hat{F}\varphi d \lambda.
	\end{align*}
	This yields that (\ref{assumplemma5}) implies (\ref{assumplemma4}).\\
	
	Finally, one obtains by using $W^{1,p}(S)\cap L^{\infty}(S) = \bigcup \limits_{k \in \mathbb{N}}(W^{1,p}(S)\cap L^{\infty}_{k}(S))$ and the equivalence of (\ref{assumplemma4}) and (\ref{assumplemma5}) that
	\begin{eqnarray*}
		& & ~
		\left\{ \omega \in \Omega:~ \int \limits_{S}g(\omega)| \nabla f(\omega)|^{p-2} \nabla f(\omega)\cdot \nabla \varphi d\lambda=\int \limits_{S}\hat{f}(\omega)\varphi d \lambda,~\forall \varphi \in W^{1,p}(S)\cap L^{\infty}(S)\right\} \\
		& = &  ~  \bigcap \limits_{ k \in \mathbb{N}} \left\{ \omega \in \Omega:~ \int \limits_{S}g(\omega)|\nabla f(\omega)|^{p-2}\nabla f(\omega)\cdot \nabla \varphi d\lambda=\int \limits_{S}\hat{f}(\omega)\varphi d \lambda,~\forall \varphi \in \mathcal{D}(k)\right\} \\
		& = &  ~  \bigcap \limits_{ k \in \mathbb{N}} ~ \bigcap \limits_{ \varphi \in \mathcal{D}(k)} \left\{ \omega \in \Omega:~ \int \limits_{S}g(\omega)|\nabla f(\omega)|^{p-2} \nabla f(\omega)\cdot \nabla \varphi d\lambda=\int \limits_{S}\hat{f}(\omega)\varphi d \lambda \right\},
	\end{eqnarray*}
	which implies, using (\ref{assumplemma7}), the claim as $\mathcal{D}(k)$ is countable for each $k \in \mathbb{N}$. 
\end{proof} 

\begin{remark} Let $q \in (1,\infty)$. Then one has $L^{q}(S) \in \mathfrak{B}(L^{1}(S))$. (This works precisely as the proof of $L^{\infty}(S) \in \mathfrak{B}(L^{1}(S))$, see Lemma  \ref{lemmawelldefined1}.)\\ 
	Now let $f \in L^{1}(\Omega;L^{1}(S))$ and assume $\P( f \in L^{q}(S))=1$. Then one has, by working as in the proof of Lemma \ref{lemmawelldefined1} that $f$ is $\F-\mathfrak{B}(L^{q}(S))$-measurable. Consequently, as the mapping $L^{q}(S) \ni h \mapsto ||h||_{L^{q}(S)}$ is continuous (and hence $\mathfrak{B}(L^{q}(S))-\mathfrak{B}(\mathbb{R})$-measurable), one has that the mapping defined by \linebreak$\Omega \ni \omega \mapsto ||f(\omega)||_{L^{q}(S)}$ is $\F-\mathfrak{B}(\mathbb{R})$-measurable.
	This holds, as the following lemma reveals, also for $q=\infty$. 
\end{remark}

\begin{lemma} Let $f \in L^{1}(\Omega;L^{1}(S))$ and assume $\P(f \in L^{\infty}(S))=1$. Then the mapping defined by $\Omega \ni \omega \mapsto ||f(\omega)||_{L^{\infty}(S)}$ is $\F-\mathfrak{B}(\mathbb{R})$-measurable.
\end{lemma}
\begin{proof} As $\P(f \in L^{\infty}(S))=1$ one has particularly $P(f \in L^{m}(S),~\forall m \in \mathbb{N})=1$. Consequently $f$ is $\F-\mathfrak{B}(L^{m}(S))$ measurable for any $m \in \mathbb{N}$. This yields that $\Omega \ni \omega \mapsto ||f(\omega)||_{L^{m}(S)}$ is $\F-\mathfrak{B}(\mathbb{R})$-measurable. Moreover, one has for $\P$-a.e. $\omega \in \Omega$ that
	\begin{align*}
	||f(\omega)||_{L^{\infty}(S)}= \lim \limits_{m \rightarrow \infty} ||f(\omega)||_{L^{m}(S)}.
	\end{align*}
	Consequently, $\Omega \ni \omega \mapsto ||f(\omega)||_{L^{\infty}(S)}$ is the almost sure limit of $\F-\mathfrak{B}(\mathbb{R})$-measurable functions and therefore itself $\F-\mathfrak{B}(\mathbb{R})$-measurable. 
\end{proof}

\section{Technical results necessary to prove the existence and uniqueness of mild solutions}
\label{appendix_b}

\begin{lemma}\label{aprop} Let $\gamma \in L^{1}_{g_{1},g_{2}}(S)$. The following assertions hold.
	\begin{enumerate}
		\item $a(\gamma)$, and consequently also $\mathfrak{a}(\gamma)$, is completely accretive.
		\item $L^{\infty}(S)\subseteq R(Id+a(\gamma))$ and consequently $R(Id+\mathfrak{a}(\gamma))=L^{1}(S)$. Hence, $\mathfrak{a}(\gamma)$ is m-accretive.
		\item $(Id+a(\gamma))^{-1}|_{L^{\infty}(S)}=(Id+\mathfrak{a}(\gamma))^{-1}|_{L^{\infty}(S)}$.
		\item  $(Id+\mathfrak{a}(\gamma))^{-1}$ is $L^{1}(S)$-continuous.
	\end{enumerate} 
\end{lemma}
\begin{proof} Proof of i). See \cite[Prop. 3.5]{main}, for the fact that $a(\gamma)$ is completely accretive. Consequently, \cite[Corollary. 2.7]{cao} implies that $\mathfrak{a}(\gamma)$ is also completely accretive.\\
	Proof of ii). See \cite[Prop. 3.5]{main} for $L^{\infty}(S)\subseteq R(Id+a(\gamma))$. Consequently, \cite[Prop. 2.18.ii)]{BenilanBook} , implies iii), because $L^{\infty}(S)$ is well known to be dense in $L^{1}(S)$.\\
	Proof of iii). Let $h \in L^{\infty}(S)$ and let $(f,\hat{f}) \in a(\gamma)$, resp. $(F,\hat{F}) \in \mathfrak{a}(\gamma)$, be the uniquely determine functions fulfilling $h=f+\hat{f}$, resp. $h=F+\hat{F}$, i.e. $f=(Id+a(\gamma))^{-1}h$, resp. $F=(Id+\mathfrak{a}(\gamma))^{-1}h$.\\
	The complete accretivity of $\mathfrak{a}(\gamma)$ yields $F<<F+\hat{F}$ and consequently $F<<h$. This implies $F \in L^{\infty}(S)$ since $h \in L^{\infty}(S)$. Hence, it follows by virtue of \cite[Lemma 3.1]{ich} that $(F,\hat{F}) \in a(\gamma)$. Conclusively, one has, by uniqueness, that $f=F$.\\ 
	Proof of iv). Resolvents are not only continuous, but even nonexpansive, cf. \cite[Prop. 4.3]{BenilanBook}.
\end{proof}

\begin{lemma}\label{fancylemma} Let $(\gamma_{m})_{m \in \mathbb{N}} \subseteq L^{1}_{g_{1},g_{2}}(S)$ and assume that there is $\gamma \in L^{1}(S)$ such that
	\begin{align*}
	\lim \limits_{m \rightarrow \infty} \gamma_{m}=\gamma,\text{ in } L^{1}(S).
	\end{align*}
	Then
	\begin{align*}
	\wlim \limits_{m \rightarrow \infty} (Id+a(\gamma_{m}))^{-1}h=(Id+a(\gamma))h,~\text{in } L^{1}(S),
	\end{align*}
	for any $h \in L^{\infty}(S)$.
\end{lemma}
The following proof is long and technical. Moreover, the proof works similar to the one of \cite[Prop. 3.5]{main} which states the range condition $L^{\infty}(S)\subseteq R(Id+a(\gamma))$. Proving the range condition works by showing that a certain resolvent converges. As in our case, it is easy to see that the resolvent converges to a limit, but it is very challenging to show that this is the correct limit. And the delicate technique which is used to show that the limit is the correct one, is the same as in \cite{main}. As it is, on a first glance, not that obvious that these proofs work similar, the proof of Lemma \ref{fancylemma} will be given here.
\begin{proof} Firstly, observe that, by passing to a subsequence if necessary, $\lim \limits_{m \rightarrow \infty}\gamma_{m}=\gamma$ a.e. on $S$. Moreover, it is clear that $|\gamma|\leq g_{2}$ a.e. on $S$. This implies $|\gamma_{m}-\gamma|^{\max(p,\tp)}\leq (2g_{2})^{\max(p,\tp)}$. Consequently, since $\lambda(S)<\infty$ it follows by virtue of dominated convergence that
	\begin{align}
	\label{fancylemmaproof7}
	\lim \limits_{m \rightarrow \infty} \gamma_{m}=\gamma \text{, in } L^{p}(S) \text{ and in } L^{\tp}(S).
	\end{align}
	Let $f_{m}:=(Id+a(\gamma_{m}))^{-1}h$ for each $m \in \mathbb{N}$ and $f:=(Id+a(\gamma))^{-1}h$. Additionally introduce\linebreak $\hat{f}_{m}:=a(\gamma_{m})f_{m}$ for all $m \in \mathbb{N}$ and $\hat{f}:=a(\gamma)f$.\\
	Note that by construction $f+\hat{f}=h=f_{m}+\hat{f}_{m}$ for each $m \in \mathbb{N}$. Moreover, one has
	\begin{align}
	\label{fancylemmaproof2}
	f_{m} << f_{m}+\hat{f}_{m}=h,~\forall m \in \mathbb{N},
	\end{align}
	by complete accretivity.\\
	Consequently, $||f_{m}||_{L^{p}(S)}\leq ||h||_{L^{p}(S)}<\infty$ for each $m \in \mathbb{N}$. As $L^{p}(S)$ is reflexive this implies, by passing to a subsequence if necessary, that there  is an $F \in L^{p}(S)$ such that
	\begin{align}
	\label{fancylemmaproof3}
	\wlim \limits_{m \rightarrow \infty} f_{m}=F \text{ in } L^{p}(S).
	\end{align}
	Now introduce $\hat{F}:=h-F$. Then 
	\begin{align}
	\label{fancylemmaproof5}
	\wlim \limits_{m \rightarrow \infty} \hat{f}_{m}=\hat{F} \text{ in } L^{p}(S),
	\end{align}
	as $\hat{f}_{m}=h-f_{m}$ for each $m \in \mathbb{N}$.\\
	Now it will be verified that 
	\begin{align}
	\label{fancylemmaproof1}
	F \in W^{1,p}(S) \cap L^{\infty}(S).
	\end{align}
	Proof of (\ref{fancylemmaproof1}). First of all it follows from (\ref{fancylemmaproof2}), (\ref{fancylemmaproof3}) and by the virtue of \cite[ Corollary 2.7]{cao} that $F<<h$ and consequently $F \in L^{\infty}(S)$. Hence, particularly $F \in L^{p}(S)$.\\
	Moreover, $(f_{m},\hat{f}_{m}) \in a(\gamma_{m})$ together with (\ref{fancylemmaproof2}) yields
	\begin{align}
	\label{fancylemmaproof4}
	||\nabla f_{m}||^{p}_{L^{p}(S;\mathbb{R}^{n})}=\int \limits_{S} |\nabla f_{m}|^{p}d \lambda \leq \frac{1}{g_{1}} \int \limits_{S} \gamma_{m} |\nabla f_{m}|^{p}d \lambda = \frac{1}{g_{1}} \int \limits_{S} f_{m}\hat{f}_{m}d \lambda \leq \frac{2}{g_{1}} \lambda(S) ||h||^{2}_{L^{\infty}(S)}.
	\end{align}
	Consequently, by passing to a subsequence if necessary, there is an $\mathbb{F}=(\mathbb{F}_{1},...,\mathbb{F}_{n}) \in L^{p}(S;\mathbb{R}^{n})$ such that $	\wlim \limits_{m \rightarrow \infty} \nabla f_{m}=\mathbb{F} \text{ in } L^{p}(S;\mathbb{R}^{n}).$ This, together with (\ref{fancylemmaproof3}), implies for all $\varphi \in C_{c}^{\infty}(S)$ that 
	\begin{align*}
		\int \limits_{S} F \frac{\partial}{\partial x_{j}} \varphi d \lambda=\lim \limits_{m \rightarrow \infty}\int \limits_{S} f_{m} \frac{\partial}{\partial x_{j}} \varphi d \lambda=\lim \limits_{m \rightarrow \infty}-\int \limits_{S}  \varphi  \frac{\partial}{\partial x_{j}} f_{m}d \lambda =-\int \limits_{S}  \varphi \mathbb{F}_{j}d \lambda,
	\end{align*}
	i.e. $F \in W^{1,1}_{\text{Loc}}(S)$ and $\nabla F=\mathbb{F}$. Consequently, (\ref{fancylemmaproof1}) holds and also
	\begin{align}
	\label{fancylemmaproof12}
	\wlim \limits_{m \rightarrow \infty} \nabla f_{m}=\nabla F \text{ in } L^{p}(S;\mathbb{R}^{n}).
	\end{align}
	Now observe that (\ref{fancylemmaproof4}) yields
	\begin{align*}
	||~|\nabla f_{m}|^{p-2}\nabla f_{m}||_{L^{\tp}(S;\mathbb{R}^{n})}^{\tp} = \int \limits_{S} |\nabla f_{m}|^{(p-1)\tp}d \lambda = \int \limits_{S} |\nabla f_{m}|^{p}d \lambda \leq \frac{2}{g_{1}} \lambda(S) ||h||^{2}_{L^{\infty}(S)}.
	\end{align*}
	Consequently, by passing to a subsequence if necessary, there is a $\zeta \in L^{\tp}(S;\mathbb{R}^{n})$ such that
	\begin{align}
	\label{fancylemmaproof8}
	\wlim \limits_{m \rightarrow \infty} |\nabla f_{m}|^{p-2}\nabla f_{m}=\zeta ~ \text{in } L^{\tp}(S;\mathbb{R}^{n}).
	\end{align}
	Now it will be proven that 
	\begin{align}
	\label{fancylemmaproof9}
	\int \limits_{S} \gamma  \zeta \cdot  \nabla \varphi d \lambda = \int \limits_{S} \hat{F} \varphi d \lambda,~\forall \varphi \in W^{1,p}(S) \cap L^{\infty}(S).
	\end{align}
	For $\varphi \in C^{\infty}(\overline{S})$, (\ref{fancylemmaproof4}) implies that
	\begin{align*}
	\int \limits_{S} \left| |\nabla f_{m}|^{p-2}  \nabla f_{m}\cdot \nabla \varphi \right|^{\tp}d \lambda \leq \int \limits_{S} |\nabla f_{m}|^{p} |\nabla \varphi|^{\tp}d \lambda \leq ||~ |\nabla \varphi|^{\tp}||_{L^{\infty}(S)}\frac{2}{g_{1}} \lambda(S) ||h||^{2}_{L^{\infty}(S)}<\infty,
	\end{align*}
	for all $m \in \mathbb{N}$.\\
	This yields, by virtue of H\"older's inequality and (\ref{fancylemmaproof7}) that
	\begin{align*}
	\lim \limits_{m \rightarrow \infty} \left| \int \limits_{S} (\gamma_{m}-\gamma)|\nabla f_{m}|^{p-2} \nabla f_{m}\cdot \nabla \varphi d \lambda \right| \leq \lim \limits_{m \rightarrow \infty} ||\gamma_{m}-\gamma||_{L^{p}(S)} ||~|\nabla f_{m}|^{p-2}  \nabla f_{m}\cdot \nabla \varphi ||_{L^{\tp}(S)}=0
	\end{align*}
	for all $\varphi \in C^{\infty}(\overline{S})$.\\
	Using this, (\ref{fancylemmaproof5}) as well as (\ref{fancylemmaproof8}) yields (\ref{fancylemmaproof9}) for $\varphi \in C^{\infty}(\overline{S})$. Moreover, for arbitrary \linebreak$\varphi \in W^{1,p}(S) \cap L^{\infty}(S)$, there is, as $S$ is of class $C^{1}$, a sequence $(\varphi_{m})_{m \in \mathbb{N}} \subseteq  C^{\infty}(\overline{S})$ such that $\lim \limits_{m \rightarrow \infty} \varphi_{m}=\varphi$ in $W^{1,p}(S)$. Hence, as $\gamma \zeta \in L^{\tp}(S;\mathbb{R}^{n})$ and $\hat{F}=h-F \in L^{\infty}(S) \subseteq L^{\tp}(S)$, one obtains
	\begin{align*}
		\int \limits_{S} \hat{F} \varphi d \lambda=\lim \limits_{m \rightarrow \infty} \int \limits_{S} \hat{F} \varphi_{m} d \lambda= \lim \limits_{m \rightarrow \infty}  \int \limits_{S} \gamma  \zeta\cdot  \nabla \varphi_{m}  d \lambda=\int \limits_{S} \gamma  \zeta\cdot  \nabla \varphi  d \lambda,
	\end{align*}
	which verifies (\ref{fancylemmaproof9}).\\
	Now observe the following: If one has
	\begin{align}
	\label{fancylemmaproof10}
	\zeta = |\nabla F|^{p-2}\nabla F,
	\end{align}
	then (\ref{fancylemmaproof9}) yields
	\begin{align*} 
	\int \limits_{S} \gamma | \nabla F|^{p-2}   \nabla F \cdot  \varphi d \lambda = \int \limits_{S} \hat{F} \varphi d \lambda,~\forall \varphi \in W^{1,p}(S) \cap L^{\infty}(S).
	\end{align*}
	This, together with (\ref{fancylemmaproof1}) implies $(F,\hat{F}) \in a(\gamma)$.\\ 
	Since it has been already established that $h=F+\hat{F}$ and since also $h=f+\hat{f}$ as well as $(f,\hat{f}) \in a(\gamma)$ the accretivity of $a(\gamma)$ yields
	\begin{align*}
	||f-F||_{L^{1}(S)} \leq ||f-F+\hat{f}-\hat{F}||_{L^{1}(S)}=||h-h||_{L^{1}(S)}=0.
	\end{align*}
	Consequently, (\ref{fancylemmaproof10}) implies $f=F$ and it follows by virtue of (\ref{fancylemmaproof3}) that
	\begin{align*}
	\wlim \limits_{m \rightarrow \infty} (Id+a(\gamma_{m}))^{-1}h=\wlim \limits_{m \rightarrow \infty} f_{m}=F=f=(Id+a(\gamma))^{-1}h~\text{in } L^{p}(S).
	\end{align*}
	Conclusively, since $L^{p}(S)$-weak convergence implies $L^{1}(S)$-weak convergence, the claim follows once (\ref{fancylemmaproof10}) is proven.\\
	The delicate proof of (\ref{fancylemmaproof10}) is preceded by the proofs of the following four statements.\\
	One has, by passing to a subsequence if necessary:
	\begin{compactenum}[(I)]
		\item  $\limsup \limits_{m \rightarrow \infty} \int \limits_{S} \gamma_{m} |\nabla f_{m}|^{p} \leq \int \limits_{S}F\hat{F} d \lambda$.
		\item $\wlim \limits_{m \rightarrow \infty}\gamma_{m}\nabla f_{m}=\gamma \nabla F$ in $L^{p}(S)$.
		\item $\wlim \limits_{m \rightarrow \infty}\gamma_{m}|\nabla f_{m}|^{p-2}\nabla f_{m}=\gamma \zeta$ in $L^{\tp}(S;\mathbb{R}^{n})$.
		\item $\lim \limits_{m \rightarrow \infty}\int \limits_{S}(\gamma-\gamma_{m})|\varphi|^{p}d\lambda=0$ for all $\varphi \in L^{p}(S;\mathbb{R}^{n})$.
	\end{compactenum}
	Proof of (I). Firstly, (\ref{fancylemmaproof2}) implies that $||f_{m}||_{L^{2}(S)}\leq||h||_{L^{2}(S)}$. Consequently, $(f_{m})_{m \in \mathbb{N}}$ has, by passing to a subsequence if necessary, an $L^{2}(S)$-weakly convergent subsequence, converging to an $\tilde{f} \in L^{2}(S)$. Moreover, it is plain that $\tilde{f}=F$ and therefore
	\begin{align}
	\label{fancylemmaproof11}
	\int \limits_{S} F^{2} d \lambda \leq \liminf \limits_{m \rightarrow \infty} \int \limits_{S} f_{m}^{2} d \lambda.
	\end{align}
	Conclusively, it follows by virtue of (\ref{fancylemmaproof3}) as well as (\ref{fancylemmaproof11}) that
	\begin{align*}
		\limsup \limits_{m \rightarrow \infty} \int \limits_{S} \gamma_{m} |\nabla f_{m}|^{p}d \lambda\leq \limsup \limits_{m \rightarrow \infty} \int \limits_{S} f_{m}hd \lambda-\liminf \limits_{m \rightarrow \infty}\int \limits_{S} f_{m}^{2}d \lambda =\int \limits_{S} F \hat{F}d \lambda.
	\end{align*} 
	\noindent Proof of (II). Firstly, note that (\ref{fancylemmaproof4}) yields
	\begin{align*}
	||\gamma_{m}\nabla f_{m}||_{L^{p}(S;\mathbb{R}^{n})}^{p} = \int \limits_{S} \gamma_{m}^{p} |\nabla f_{m}|^{p} d\lambda \leq \frac{g_{2}^{p}}{g_{1}} 2 \lambda(S) ||h||_{L^{\infty}(S)}^{2}<\infty,~\forall m \in \mathbb{N}.
	\end{align*}
	Consequently, by passing to a subsequence if necessary, there is an $\alpha \in L^{p}(S;\mathbb{R}^{n})$ such that
	\begin{align}
	\label{fancylemmaproof13}
	\wlim \limits_{m \rightarrow \infty} \gamma_{m}\nabla f_{m} = \alpha \text{ in } L^{p}(S;\mathbb{R}^{n}).
	\end{align}
	Moreover, one has for any $\varphi \in L^{\infty}(S;\mathbb{R}^{n})$, by virtue of H\"older's inequality, Cauchy-Schwarz' inequality, (\ref{fancylemmaproof4}) and (\ref{fancylemmaproof7}) that
	\begin{align*}
		\lim \limits_{m \rightarrow \infty}\left| \int \limits_{S} (\gamma_{m}-\gamma)  \nabla f_{m}\cdot \varphi d \lambda\right|\leq  \lim \limits_{m \rightarrow \infty} ||\gamma_{m}-\gamma||_{L^{\tp}(S)} ||~|\varphi|~||_{L^{\infty}(S)}  \left( \frac{2}{g_{1}} \lambda(S)||h||_{L^{\infty}(S)} \right)^{\frac{1}{p}}  = 0.
	\end{align*}
	Consequently, one obtains for any $\varphi \in L^{\infty}(S;\mathbb{R}^{n})$ by invoking (\ref{fancylemmaproof12}) and (\ref{fancylemmaproof13}) that 
	\begin{align*}
		\int \limits_{S}  (\alpha -\gamma \nabla F)\cdot \varphi  d \lambda=\lim \limits_{m \rightarrow \infty} \int \limits_{S}  \gamma_{m}\nabla f_{m}\cdot \varphi  - \nabla f_{m}\cdot   \varphi\gamma  d \lambda= 0.
	\end{align*}
	This clearly implies $\alpha=\gamma \nabla F$. Hence, (\ref{fancylemmaproof13}) implies (II).\\
	Proof of (III). Firstly, note that it follows by virtue of (\ref{fancylemmaproof4}) that 
	\begin{align*}
	||\gamma_{m}|\nabla f_{m}|^{p-2}\nabla f_{m}||_{L^{\tp}(S;\mathbb{R}^{n})}^{\tp} \leq g_{2}^{\tp} \int \limits_{S}  |\nabla f_{m}|^{p}d \lambda \leq \frac{g_{2}^{\tp}}{g_{1}} 2 \lambda(S)||h||_{L^{\infty}(S)}^{2}.
	\end{align*}
	Consequently there is, by passing to a subsequence if necessary, an $\alpha \in L^{\tp}(S;\mathbb{R}^{n})$ such that
	\begin{align*}
	\wlim \limits_{m \rightarrow \infty} \gamma_{m}|\nabla f_{m}|^{p-2}\nabla f_{m}=\alpha \text{ in } L^{\tp}(S;\mathbb{R}^{n}).
	\end{align*}
	Moreover, one has for any $\varphi \in L^{\infty}(S;\mathbb{R}^{n})$, by virtue of H\"older's inequality, Cauchy-Schwarz' inequality, (\ref{fancylemmaproof4}) and (\ref{fancylemmaproof7}) that
	\begin{align*}
		\lim \limits_{m \rightarrow \infty} \left| \int \limits_{S} (\gamma_{m}-\gamma) |\nabla f_{m}|^{p-2}\nabla f_{m}\cdot \varphi d \lambda \right|=  0.
	\end{align*}
	Consequently, one obtains for any $\varphi \in L^{\infty}(S;\mathbb{R}^{n})$ by recalling (\ref{fancylemmaproof8}) that
	\begin{align*}
	\int \limits_{S}  (\alpha - \gamma \zeta)\cdot \varphi  d \lambda = \lim \limits_{m \rightarrow \infty} \int \limits_{S} ( \gamma_{m}|\nabla f_{m}|^{p-2}\nabla f_{m} - \gamma |\nabla f_{m}|^{p-2}\nabla f_{m})\cdot \varphi  d \lambda  = 0.
	\end{align*}
	This implies $ \alpha = \gamma \zeta$.\\
	Proof of (IV). Since particularly $\lim \limits_{m \rightarrow \infty}\gamma_{m}-\gamma=0$ in $L^{1}(S)$ one has, by passing to a subsequence if necessary, that $\lim \limits_{m \rightarrow \infty} (\gamma-\gamma_{m})|\varphi|^{p}=0 \text{ a.e. on } S$ for any given $\varphi \in L^{p}(S;\mathbb{R}^{n})$. Since plainly\linebreak $|(\gamma-\gamma_{m})|\varphi|^{p}|\leq 2g_{2}|\varphi|^{p} \in L^{1}(S)$, dominated convergence yields (IV).\\
	Proof of (\ref{fancylemmaproof10}). Let $\varphi \in L^{p}(S;\mathbb{R}^{n})$. First of all it is a direct consequence of Cauchy-Schwarz' inequality for the Euclidean norm that
	\begin{align*}
	\gamma_{m} ( |\nabla f_{m}|^{p-2}\nabla f_{m}-|\varphi|^{p-2}\varphi)\cdot(\nabla f_{m}-\varphi ) \geq 0,~\forall m \in \mathbb{N}
	\end{align*}
	and consequently
	\begin{align*}
	\int \limits_{S} \gamma_{m} |\varphi|^{p-2}  \varphi\cdot(\nabla f_{m}-\varphi )d \lambda \leq \int \limits_{S} \gamma_{m} |\nabla f_{m}|^{p-2} \nabla f_{m}\cdot(\nabla f_{m}-\varphi )d \lambda,~\forall m \in \mathbb{N},~\varphi \in L^{p}(S;\mathbb{R}^{n}). 
	\end{align*}
	(Hereby the existence of both integrals is a direct consequence of the boundedness of $\gamma_{m}$ and H\"older's inequality.)\\
	The last yields, by using at first (II) and (IV), then the last inequality, and finally (I) as well as (III) that 
	\begin{align}
	\label{fancylemmaproof15}
	\int \limits_{S} \gamma |\varphi|^{p-2} \varphi\cdot(\nabla F-\varphi )d\lambda \leq \int \limits_{S} F\hat{F}-\gamma  \zeta\cdot \varphi d \lambda,~\forall \varphi \in L^{p}(S;\mathbb{R}^{n}),
	\end{align}
	Now note that it follows from $f_{m} \in D(a(\gamma_{m}))$ that particularly $f_{m} \in W^{1,p}(S) \cap L^{\infty}(S)$. Consequently, by (\ref{fancylemmaproof9}) one gets
	\begin{align}
	\label{fancylemmaproof14}
	\int \limits_{S} \gamma  \zeta \cdot  \nabla f_{m} d \lambda = \int \limits_{S} \hat{F} f_{m} d \lambda,~\forall m \in \mathbb{N}.
	\end{align}
	Now note that combining (\ref{fancylemmaproof3}), (\ref{fancylemmaproof12}) and (\ref{fancylemmaproof14}) yields
	\begin{align*} 
	\int \limits_{S} \gamma  \zeta \cdot  \nabla F d \lambda = \int \limits_{S}\hat{F}F d\lambda
	\end{align*}
	Consequently, one infers from (\ref{fancylemmaproof15}) that
	\begin{align}
	\label{fancylemmaproof16}
	\int \limits_{S} \gamma |\varphi|^{p-2} \varphi\cdot (\nabla F-\varphi )d\lambda \leq \int \limits_{S} \gamma  \zeta \cdot ( \nabla F-\varphi ) d \lambda,~\forall \varphi \in L^{p}(S;\mathbb{R}^{n}).
	\end{align}
	Now note that $\nabla F \in L^{p}(S;\mathbb{R}^{n})$ which implies that $\nabla F-\alpha \varphi$ is, for any $\alpha \in (0,\infty)$ and $\varphi \in L^{p}(S;\mathbb{R}^{n})$ a valid choice as a test function in (\ref{fancylemmaproof16}). Hence, using $\nabla F-\alpha \varphi$ as a test function in (\ref{fancylemmaproof16}) and dividing the resulting equation by $\alpha$ yields
	\begin{align}
	\label{fancylemmaproof17}
	\int \limits_{S} \gamma |\nabla F-\alpha\varphi|^{p-2}( \nabla F-\alpha\varphi)\cdot\varphi d\lambda \leq \int \limits_{S} \gamma  \zeta \cdot\varphi  d \lambda,~\forall \varphi \in L^{p}(S;\mathbb{R}^{n}),~\alpha \in (0,\infty).
	\end{align}
	It is obvious that
	\begin{align*}
	\lim \limits_{\alpha \downarrow 0} \gamma |\nabla F-\alpha\varphi|^{p-2}( \nabla F-\alpha\varphi)\cdot \varphi  = \gamma |\nabla F|^{p-2} \nabla F\cdot\varphi  \text{ a.e. on } S
	\end{align*}
	for a given $\varphi \in L^{p}(S;\mathbb{R}^{n})$. Now let $\varepsilon>0$ and $\alpha \in (0,\varepsilon)$. Then one instantly verifies that
	\begin{align*}
	|\gamma |\nabla f-\alpha\varphi|^{p-2}( \nabla F-\alpha\varphi)\cdot\varphi | \leq g_{2} (|\nabla F|+\varepsilon|\varphi|)^{p-1}|\varphi| \text{ a.e. on } S
	\end{align*}
	for any $\varphi \in L^{p}(S;\mathbb{R}^{n})$. Now it is a direct consequence of H\"older's inequality that the right-hand-side of the last inequality is in $L^{1}(S)$ for any $\varphi \in L^{p}(S;\mathbb{R}^{n})$. Hence, it follows by virtue of dominated convergence that
	\begin{align*}
	\lim \limits_{\alpha \downarrow 0} \int \limits_{S} \gamma |\nabla F-\alpha\varphi|^{p-2}( \nabla f-\alpha\varphi)\cdot \varphi d\lambda = \int \limits_{S} \gamma |\nabla F|^{p-2} \nabla F\cdot \varphi  d \lambda.
	\end{align*}
	Consequently, it follows by recalling (\ref{fancylemmaproof17}) that 
	\begin{align*} 
	\int \limits_{S} \gamma |\nabla F|^{p-2} \nabla F\cdot \varphi d\lambda \leq \int \limits_{S} \gamma  \zeta \cdot  \varphi  d \lambda,~\forall \varphi \in L^{p}(S;\mathbb{R}^{n}).
	\end{align*}
	Conclusively, replacing $\varphi$ by $-\varphi$ implies
	\begin{align*} 
	\int \limits_{S}  ( \gamma |\nabla F|^{p-2}\nabla F-  \gamma\zeta )\cdot\varphi  d \lambda=0,~\forall \varphi \in L^{p}(S;\mathbb{R}^{n}).
	\end{align*}
	Finally, this yields $\gamma |\nabla F|^{p-2}\nabla F-  \gamma\zeta=0$ a.e. on $S$ which implies (\ref{fancylemmaproof10}) since particularly $\gamma\neq 0$ a.e. on $S$.
\end{proof}

\noindent The following lemma can be inferred from  \cite[Lemma 5.6]{ich} and the first lines of its proof:

\begin{lemma}\label{deterdomdenselemma} Let  $\gamma \in L^{1}_{g_{1},g_{2}}(S)$, $h \in L^{\infty}(S)$ and introduce $f_{m}:=(Id + \frac{1}{m}a(\gamma))^{-1}h$. Then one has
	\begin{align*}
	\lim \limits_{m \rightarrow \infty} f_{m}=h,~\text{in } L^{1}(S).
	\end{align*}
	Particularly, $D(a(\gamma))$ as well as $D(\mathfrak{a}(\gamma))$ are dense in $L^{1}(S)$.
\end{lemma}

\begin{lemma}\label{denselemma} $\tau(L^{1}(\Omega;L^{1}(S)))$ as well as $L^{1,\infty}(\Omega;L^{1}(S))$ are dense subsets of $L^{1}(\Omega;L^{1}(S))$.
\end{lemma}
\begin{proof} Firstly, note that clearly $\tau(L^{1}(\Omega;L^{1}(S))) \subseteq L^{1,\infty}(\Omega;L^{1}(S))$ which implies that it suffices to prove the claim for $\tau(L^{1}(\Omega;L^{1}(S)))$.\\ 
	Now let $f \in L^{1}(\Omega;L^{1}(S))$ and introduce $f_{k}:=\tau_{k}(f)$ for each $k \in \mathbb{N}$.\\
	As $\lim \limits_{k \rightarrow \infty}\tau_{k}(s)=s$ for each $s \in \mathbb{R}$ it is clear that $\lim \limits_{k \rightarrow \infty}f_{k}(\omega)=f(\omega)$ a.e. on $S$ for every given $\omega \in \Omega$, up to a $\P$-nullset.\\
	Since $|f_{k}(\omega)-f(\omega)| \leq 2 |f(\omega)| \in L^{1}(S),$ Lebesgue's theorem yields
	\begin{align*}
	\lim \limits_{k \rightarrow \infty} ||f_{k}(\omega)-f(\omega)||_{L^{1}(S)}=0 \text{, for $\P$-a.e. } \omega \in \Omega.
	\end{align*}
	Moreover, $||f_{k}(\cdot)-f(\cdot)||_{L^{1}(S)} \leq 2 ||f(\cdot)||_{L^{1}(S)} \in L^{1}(\Omega),$ which implies, by applying Lebesgue's Theorem again, that $\lim \limits_{k \rightarrow \infty} f_{k}=f$ in $L^{1}(\Omega;L^{1}(S))$. 
\end{proof}

\end{document}